\documentclass[reqno]{amsart}

\usepackage[usenames,dvipsnames]{xcolor}

\usepackage{amssymb,amscd,enumerate,mathrsfs,tikz} 
\usepackage{esint} 

\numberwithin{equation}{section}
\newtheorem{theorem}[equation]{Theorem}
\newtheorem{proposition}[equation]{Proposition}
\newtheorem{lemma}[equation]{Lemma}
\newtheorem{corollary}[equation]{Corollary}

\theoremstyle{definition}
\newtheorem{definition}[equation]{Definition}

\newtheorem{example}[equation]{Example}

\DeclareMathOperator{\Diff}{Diff}

\DeclareMathOperator{\Gr}{Gr}

\DeclareMathOperator{\LL}{Hom}
\DeclareMathOperator{\Mellin}{M}

\DeclareMathOperator{\rg}{rg}

\DeclareMathOperator{\spec}{spec}
\DeclareMathOperator{\supp}{supp}
\DeclareMathOperator{\sym}{sym}

\DeclareMathOperator{\Span}{span}

\def\a{\mathfrak a}

\def\Hol{\mathfrak{H}}

\def\m{\mathfrak m}
\def\Mero{\mathfrak{M}}

\def\C{\mathbb C}
\def\ZN{\mathbb N}

\def\R{\mathbb R}

\def\A{\mathcal A}
\def\B{\mathcal B}
\def\Dom{\mathcal D}
\def\Sing{\mathcal E}

\def\G{\mathcal G}
\def\H{\mathcal H}

\def\Jay{\mathcal I}

\def\M{\mathcal M}
\def\Z{\mathcal N}

\def\P{\mathcal P}

\def\Ring{\mathcal R}
\def\gSing{\mathcal S}

\def\Z{\mathcal Z}

\def\SingStar{\mathfrak E}
\def\DomStar{\mathfrak D}
\def\rr{\mathfrak r}
\def\ss{\mathfrak s}

\def\eps{\varepsilon}
\def\minus{\backslash}
\def\im{i}
\def\Id{I}

\def\bjay{\pmb \jmath}

\def\from{\leftarrow}

\def\open#1{\smash[t]{\overset{{}_{\,\,\circ}}{#1}{}}}

\def\set#1{\{#1\}}
\def\Set#1{\Bigl\{#1\Bigr\}}

\def\ghost#1{}

\def\rpar{)}
\def\lbra{[}

\def\csym{\,{}^c\!\sym}

\def\cT{\,{\hspace*{0.4pt}}{}^c{\hspace*{-0.4pt}}T}

\def\Wedge{\raise2ex\hbox{$\mathchar"0356$}}

\def\bP{\,{}^b\!P}

\def\bPar#1{\,{}^b\!(#1)}

\def\ointc{\ointctrclockwise}
\def\ointr{\ointclockwise}

\def\contract{\mathbin{\rfloor}}

\def\clap#1{\hbox to 0pt{\hss#1\hss}}
\def\mathclap{\mathpalette\mathclapinternal}
\def\mathclapinternal#1#2{\clap{$\mathsurround=0pt#1{#2}$}}

\begin{document}

\title[Elliptic complexes of first-order cone operators]{Elliptic complexes of first-order cone operators: ideal boundary conditions}
\author{Thomas Krainer}
\address{Penn State Altoona\\ 3000 Ivyside Park \\ Altoona, PA 16601-3760}
\email{krainer@psu.edu}
\author{Gerardo A. Mendoza}
\address{Department of Mathematics\\ Temple University\\ Philadelphia, PA 19122}
\email{gmendoza@temple.edu}

\begin{abstract}
The purpose of this paper is to provide a detailed description of the spaces that can be specified as $L^2$ domains for the operators of a first order elliptic complex on a compact manifold with conical singularities. This entails an analysis of the nature of the minimal domain and of a complementary space in the maximal domain of each of the operators. The key technical result is the nondegeneracy of a certain pairing of cohomology classes associated with the indicial complex. It is further proved that the set of choices of domains leading to Hilbert complexes in the sense of Br\"uning and Lesch form a variety, as well as a theorem establishing a necessary and sufficient condition for the operator in a given degree to map its maximal domain into the minimal domain of the next operator. 
\end{abstract}


\maketitle

\section{Introduction}

Differential operators on a manifold with conical singularities are (modeled as) operators on a smooth manifold $\M$ with smooth boundary $\Z$, of the form $x^{-k}P$, where $P\in \Diff^k_b(\M;E,F)$. Here and elsewhere $x$ denotes a fixed defining function for $\Z$ which is positive in $\open \M$; $E$, $F\to \M$ are vector bundles, and $\Diff^k_b(\M;E,F)$ is the space of $b$-operators of order $k$ of Melrose \cite{RBM2}. These are the operators $P\in \Diff^k(\M;E,F)$ singled out by the property that $x^{-\nu}Px^\nu\in \Diff^k(\M;E,F)$ for any $\nu$.

We shall be dealing with a  complex
\begin{multline}\label{TheCComplex}
0\to C_c^\infty(\open\M;E^{0})\xrightarrow{A_0} C_c^\infty(\open\M;E^1)\to\cdots\\
\dots\to  C_c^\infty(\open\M;E^{m-1})\xrightarrow{A_{m-1}}C_c^\infty(\open\M;E^m)\to 0
\end{multline}
of first order cone operators $A_q \in x^{-1}\Diff^1_b(\M;E^q,E^{q+1})$ acting on sections of vector bundles $E^q\to\M$. The manifold $\M$ is assumed to be compact and the complex $c$-elliptic, which as in the regular case means that its $c$-symbol sequence is exact. The notion of $c$-symbol is an adaptation to cone operators of the regular symbol of a differential operator. It is defined on the $c$-cotangent bundle of $\M$, the vector bundle whose smooth sections are the smooth $1$-forms on $\M$ which are conormal to $\Z$. For details on the $c$-symbol see \cite[Section 3]{GiKrMe07}.

Fix a smooth positive $b$-density $\m_b$ on $\M$, i.e., $x\m_b$ is a positive density on $\M$, a Hermitian metric on each of the vector bundles $E^q$, and a number $\gamma\in \R$, and view $A_q$ as an unbounded operator
$$
A_q : C_c^{\infty}(\open\M;E^q) \subset x^{-\gamma}L^2_b(\M;E^q) \to x^{-\gamma}L^2_b(\M;E^{q+1}).
$$
As such, its maximal domain is 
$$
\Dom_{\max}^q = \set{u \in x^{-\gamma}L^2_b(\M;E^q);\; A_q u \in x^{-\gamma}L^2_b(\M;E^{q+1})},
$$
a Hilbert space with the graph inner product
\begin{equation}\label{InnerProduct}
\langle u,v \rangle_q = \langle u,v \rangle_{x^{-\gamma}L^2_b} + \langle A_qu,A_qv \rangle_{x^{-\gamma}L^2_b}.
\end{equation}
Let $\Dom_{\min}^q$ be the closure of $C_c^{\infty}(\open\M;E^q)$ in $\Dom_{\max}^q$ with respect to the inner product \eqref{InnerProduct}. The spaces $\Dom_{\max}^q$ and $\Dom_{\min}^q$ are the domains of the maximal and minimal extensions $A_{q,\max}$ and $A_{q,\min}$ of $A_q$ in $x^{-\gamma}L^2_b$, respectively.

The following definition goes back to Br\"uning and Lesch \cite[Section~3]{BrueningLesch1992} for general elliptic complexes and was inspired by the work of Cheeger \cite{Cheeger1979,Cheeger1980,Cheeger1983}.

\begin{definition}
By an ideal boundary condition for the complex \eqref{TheCComplex} in $x^{-\gamma}L^2_b$ one means a choice of closed domains $\Dom_{\min}^q \subset \Dom^q \subset \Dom_{\max}^q$, $q=0,\dotsc,m-1$, for each $A_q$ such that $A_q(\Dom^q) \subset \Dom^{q+1}$. (By default, $\Dom^m=x^{-\gamma}L^2_b(\M;E^m)$.)
\end{definition}

There are always the canonical choices $\Dom^q = \Dom_{\max}^q$ for all $q$, or $\Dom^q = \Dom_{\min}^q$ for all $q$, see \cite[Lemma~3.1]{BrueningLesch1992}. These are referred to, respectively, as the absolute and relative complexes. We say that there is uniqueness of ideal boundary conditions in degree $q$ if $\Dom_{\max}^q = \Dom_{\min}^q$.

Part (\ref{bulletFiniteDim}) of Proposition~\ref{ComplementsDegreeq} below asserts that $\Dom_{\max}^q/\Dom_{\min}^q$ is finite-dimensional. As a consequence, $A_q$ is closed when given any domain $\Dom^q$ containing $\Dom_{\min}^q$ and contained in $\Dom_{\max}^q$. In particular, every domain $\Dom^q$ is of the form $\Dom_{\min}^q+D^q$ where $D^q$ is a finite-dimensional vector space that can be specified uniquely as a subspace of the orthogonal complement, $\Sing^q$, of $\Dom_{\min}^q$ in $\Dom_{\max}^q$. 

\smallskip
The task is to elucidate, on the one hand, the nature of the space $\Dom_{\min}^q$,  a subspace of every closed extension of $A_q$, and on the other, the nature of the spaces $\Sing^q$, equivalently, the quotients $\Dom_{\max}^q/\Dom_{\min}^q$. Focusing for the moment on the former, note that because the operators $A_q$ are by themselves not elliptic, the spaces $\Dom_{\min}^q$ should not be expected to be weighted $b$-Sobolev spaces (see \cite{RBM2}) or simple variants thereof, in contrast with the case of a single elliptic cone operator \cite[Proposition 3.6]{GiMe01}. But a somewhat more elaborate statement does hold true:

\begin{theorem}
Let $u \in \Dom_{\min}^q$. Then there exists $v \in \Dom^{q-1}_{\min}$ such that
\begin{equation*}
u - A_{q-1}v \in \bigcap_{\eps > 0}x^{1-\gamma-\eps}H^1_b(M;E^q).
\end{equation*}
\end{theorem}

The proof is given in Section~\ref{DminRegularity}.

\smallskip
Concerning $\Dom_{\max}^q/\Dom_{\min}^q$, our result is as follows. First, let $P_q=xA_q$. The Taylor expansion of $P_q$ along $\Z$ is $P_q=\sum_k x^kP_q^{(k)}$ where the $P_q^{(k)}$ have coefficients independent of $x$, which means that $P_q^{(k)}xD_x=xD_x P_q^{(k)}$. More properly this ought to be defined using a tubular neighborhood map and connections as in \cite[Pg. 748]{GiKrMe07}, and the resulting operators defined as living on the total space of $\pi_\wedge:\Z^\wedge\to\Z$, the inward pointing normal bundle of $\Z$ in $\M$; this will be reflected in the notation. Let 
\begin{equation*}
A_q^{(k)}=x^{-1}P_q^{(k)}\in x^{-1}\Diff^1_b(\Z^\wedge;E^q_\Z,E^{q+1}_\Z).
\end{equation*}
Then in particular $A_{q+1}^{(0)}A_q^{(0)}=0$. Let $\gSing_{\sigma_0}(\Z^\wedge;E^q_\Z)$ be the space of sections of $\pi_\wedge^*E_\Z\to \Z^\wedge$ of the form
\begin{equation*}
u(x,z)=\sum_{\ell=0}^N u_\ell(z)\, x^{\im \sigma_0}\log^\ell x,\quad u_\ell\in C^\infty(\Z;E^q_\Z)
\end{equation*}
for some $N$. Then $A_q^{(0)}$ maps $\gSing_{\sigma_0}(\Z^\wedge;E^q_\Z)$ into $\gSing_{\sigma_0+\im}(\Z^\wedge;E^{q+1}_\Z)$ and one has a complex 
\begin{equation}\tag{\ref{gSingComplex}}
\cdots \to
\gSing_{\sigma_0-\im}(\Z^\wedge;E^{q-1}_\Z)\xrightarrow{A_{q-1}^{(0)}}
\gSing_{\sigma_0}(\Z^\wedge;E^q_\Z)\xrightarrow{A_{q}^{(0)}}
\gSing_{\sigma_0+\im}(\Z^\wedge;E^{q+1}_\Z)\to \cdots
\end{equation}
Let $\H^q_{\sigma_0}(\Z;A)$ be its cohomology in degree $q$. Let
\begin{equation*}
\Sigma=\set{\sigma\in \C:\gamma-1<\Im\sigma<\gamma}.
\end{equation*}
Then:

\begin{theorem}\label{Theorem.A}
For each $q$ the space $\H^q_{\sigma_0}(\Z;A)$ is finite-dimensional for all $\sigma_0 \in \Sigma$, nonzero only for finitely many $\sigma_0 \in \Sigma$. Let $\Sigma_q^\gamma\subset \Sigma$ be the set of these $\sigma_0$. There exists a canonical isomorphism
\begin{equation*}
\Sing^q\cong \bigoplus_{\mathclap{\sigma_0\in \Sigma_q^\gamma}}\H^q_{\sigma_0}(\Z;A).
\end{equation*}
\end{theorem}

We discuss other basic properties of the spaces $\Sing^q$ in Sections~\ref{IBC} and \ref{gSing}.

The proof of the theorem proceeds along the following lines. Objects on $\Z^\wedge$ near the zero section are identified with similar objects on $\M$ near $\Z$. Pick $\omega\in C^\infty(\M)$, equal to $1$ near $\Z$ and with support as close to $\Z$ as necessary. We first prove that the elements of $\Sing^q$ are of the form
\begin{equation*}
\omega\sum_{\mathclap{\sigma_0\in \Sigma_q'}} u_{\sigma_0} + u'
\end{equation*}
for a finite subset $\Sigma_q' \subset \Sigma$ with $u'\in \Dom_{\min}^q$,  $u_{\sigma_0}\in \gSing_{\sigma_0}(\Z^\wedge;E^q_\Z)$, and each $\omega u_{\sigma_0} \in \Dom_{\max}^q$. It is easy to see that if $u_{\sigma_0}\in \gSing_{\sigma_0}(\Z^\wedge;E^q_\Z)$ is not zero, then $\omega u_{\sigma_0}\in x^{-\gamma}L^2_b(\M;E^q)$ if and only if $\Im \sigma_0<\gamma$. Using the Taylor expansion of $A_q$ one checks that $A_q\omega u_{\sigma_0}\in x^{-\gamma}L^2_b(\M;E^{q+1})$ if $\Im\sigma_0<\gamma-1$ or $A_q^{(0)}u_{\sigma_0}=0$. This links the closed elements of \eqref{gSingComplex} in degree $q$ with the maximal domain of $A_q$. However, note that if $u_{\sigma_0}$ is exact,
\begin{equation*}
u_{\sigma_0}=A_{q-1}^{(0)} w_{\sigma_0-\im}
\end{equation*}
for some $w_{\sigma_0-\im}\in \gSing_{\sigma_0-\im}(\Z^\wedge;E^{q-1}_\Z)$, then $\omega u_{\sigma_0}\in \Dom_{\min}^q$ because, as is easy to prove, already $\omega w_{\sigma_0-\im}\in \Dom_{\min}^{q-1}$. This is why one should remove exact elements from consideration. Of course one should also dismiss elements like $u'$ that are already in $\Dom_{\min}^q$. Thus $\Sigma'_q$ reduces to the set $\Sigma_q^{\gamma}$, and the isomorphism of the theorem is the map
$$
\Sing^q \ni \omega\sum_{\mathclap{\sigma_0\in \Sigma_q^{\gamma}}} u_{\sigma_0} + u' \mapsto
\sum_{\mathclap{\sigma_0\in \Sigma_q^{\gamma}}} {\mathbf u}_{\sigma_0} \in \bigoplus_{\mathclap{\sigma_0\in \Sigma_q^\gamma}}\H^q_{\sigma_0}(\Z;A),
$$
where ${\mathbf u}_{\sigma_0}$ is the cohomology class of $u_{\sigma_0}$. To show that this map is well-defined and surjective we need to consider elements of the form $\omega\sum_{\sigma_0\in \Sigma_q^\gamma}u_{\sigma_0}$ in $\Dom_{\min}^q$ and prove that each ${\mathbf u}_{\sigma_0} = {\mathbf 0}$ in cohomology. This is achieved through the following theorems, the main technical results of the paper.

For the formal adjoint complex of \eqref{TheCComplex}, let $\DomStar_{\max}^{q+1}$, $\DomStar_{\min}^{q+1}$ denote the maximal and minimal domains of $A_q^\star$ and $\SingStar^{q+1}$ the orthogonal complement of $\DomStar_{\min}^{q+1}$ in $\DomStar_{\max}^{q+1}$ with respect to the graph inner product defined by $A_q^\star$. The Green pairing
\begin{equation*}
[\cdot,\cdot]_{A_q} : \Dom_{\max}^q \times \DomStar_{\max}^{q+1} \to \C, \quad[u,v]_{A_q} = \langle A_qu,v \rangle_{x^{-\gamma}L^2_b} - \langle u,A_q^{\star}v \rangle_{x^{-\gamma}L^2_b}
\end{equation*}
gives a nondegenerate pairing $\Sing^q\times \SingStar^{q+1}\to \C$, reflected in cohomology in the following theorem.

\begin{theorem}\label{Theorem.B}
Let 
\begin{equation*}
u=\sum_{\mathclap{\sigma_0\in \Sigma_q^\gamma}} u_{\sigma_0},\quad u_{\sigma_0}\in \gSing_{\sigma_0}(\Z^\wedge;E^q_\Z), \quad A_q^{(0)}u_{\sigma_0}=0,
\end{equation*}
so $\omega u\in \Dom_{\max}^q$. Let also $A_q^{\star(0)}$ and $\Sigma_{q+1}^{\star \gamma}$ be the analogue of $A_q^{(0)}$ and $\Sigma_q^\gamma$ for the adjoint complex, in degree $q+1$, and let
\begin{equation*}
v=\sum_{\sigma_0'\in \Sigma_{q+1}^{\star \gamma}} w_{\sigma'_0},\quad w_{\sigma'_0}\in \gSing_{\sigma'_0}(\Z^\wedge;E^{q+1}_\Z), \quad A_q^{\star(0)}w_{\sigma'_0}=0,
\end{equation*}
so $\omega v\in \DomStar_{\max}^{q+1}$. Then
\begin{equation*}
[\omega u,\omega v]_{A_q}=\sum_{\mathclap{\sigma_0\in \Sigma_q^\gamma}} [\omega u_{\sigma_0},\omega v_{\sigma_0^\star}]_{A_q}
\end{equation*}
where $\sigma_0^\star=\overline{\sigma_0-\im(2\gamma-1)}$ is reflection of $\sigma_0$ across the line $\Im\sigma=\gamma-1/2$. 
\end{theorem}

In other words, the pairing only relates points in $\Sigma$ lying symmetrically across the line $\Im\sigma=\gamma-1/2$.

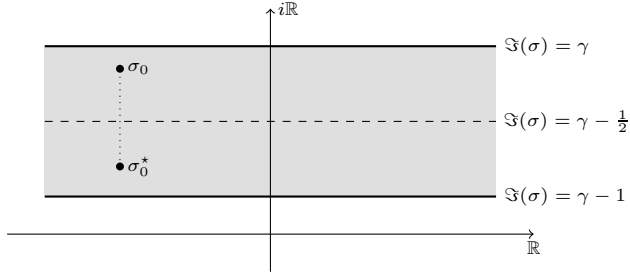
\begin{center}
\begin{tikzpicture}\scriptsize
\fill[color=lightgray,opacity=0.5] (-3,0.5) -- (3,0.5) -- (3,2.5) -- (-3,2.5) -- (-3,0.5);
\draw[->] (0,-0.5) -- (0,3) node[right] {$\im\R$};
\draw[->] (-3.5,0) -- (3.5,0) node[below] {$\R$};
\draw[thick] (-3,2.5) -- (3,2.5) node[right] {$\Im(\sigma) = \gamma$};
\draw[thick] (-3,0.5) -- (3,0.5) node[right] {$\Im(\sigma) = \gamma-1$};
\draw[dashed] (-3,1.5) -- (3,1.5) node[right] {$\Im(\sigma) = \gamma-\frac{1}{2}$};
\fill (-2,2.2) circle (1.5pt) node[right] {$\sigma_0$};
\fill (-2,0.9) circle (1.5pt) node[right] {$\sigma_0^\star$};
\draw[dotted] (-2,0.9) -- (-2,2.2);
\end{tikzpicture}
\end{center}

\begin{theorem}\label{Theorem.C}
The Green pairing induces a nonsingular pairing 
\begin{equation*}
\H^q_{\sigma_0}(\Z;A)\times \H^{q+1}_{\sigma_0^\star}(\Z;A^\star)\to \C.
\end{equation*}
for each $\sigma_0\in \Sigma$.
\end{theorem}

These theorems allow us to finish the proof of Theorem~\ref{Theorem.A}. If $\omega\sum_{\sigma_0\in \Sigma_q^\gamma}u_{\sigma_0} \in \Dom_{\min}^q$, $u_{\sigma_0} \in \gSing_{\sigma_0}(\Z^{\wedge};E^q_\Z)$, then it follows from both theorems that each $\omega u_{\sigma_0}$ belongs to $\Dom_{\min}^q$, hence $[\omega u_{\sigma_0},\omega v]_{A_q}=0$ for every $v$ as in Theorem~\ref{Theorem.B}. Therefore the class of every $u_{\sigma_0}$ is zero by Theorem~\ref{Theorem.C}, which means that $u_{\sigma_0}$ is exact.

Theorem~\ref{Theorem.A} is proved in Section~\ref{gSing} (as Theorem~\ref{Theorem.Abis})  assuming the validity of Theorems~\ref{Theorem.B} and \ref{Theorem.C}. For the proofs of these theorems we transition the analysis from physical space to Mellin space. The corresponding results on the Mellin transform side are Theorem~\ref{AdjointPairingProps} and Theorem~\ref{MainTheorem}. The proof of Theorem~\ref{Theorem.C} is the most delicate aspect of the analysis of the spaces $\Sing^q$. We first prove a finite dimensional case (Section~\ref{FiniteDimCase}), then show in Section~\ref{Reduction} how to reduce to the case of Section~\ref{FiniteDimCase} which ends our analysis of the individual spaces $\Sing^q$.

Section~\ref{DminRegularity} collects our results on the regularity of elements of $\Dom_{\min}^q$. 

In Section~\ref{Variety} we show that the collection of ideal boundary conditions for the complex \eqref{TheCComplex} is a union of algebraic varieties of (possibly) various dimensions.

Finally, in Section~\ref{SecondaryCohomology}, we give necessary and sufficient conditions for the operator $A_q$ to map $\Dom_{\max}^q$ into $\Dom_{\min}^{q+1}$, thus removing domain requirements in degree $q$ on the possible Hilbert complexes associated to \eqref{TheCComplex} in $x^{-\gamma}L^2_b$.

\medskip
The initial impetus to the systematic study of complexes on manifolds with singularities in the $C^\infty$ category was given by Cheeger's papers \cite{Cheeger1979,Cheeger1980} concerning the de Rham complex. Since then there have been numerous developments towards a theory for general single elliptic operators as well as for complexes on such manifolds, early on by  Melrose \cite{RBM2} and Schulze \cite{Schulze1988, SchuNH}. The issue of having to choose domains comes up, for instance, in the case of the de Rham complex on non-Witt spaces in the papers \cite{HunsickerMazzeo2005} by Hunsicker and Mazzeo, and, more generally, in \cite{AlbinLeichtMazzPiazza13} by  Albin, Leichtnam, Mazzeo, and Piazza. Without attempting to be exhaustive, domains have entered explicitly in the specific case of conical singularities in work by, for example, Bei \cite{Bei2013}, Br\"uning and Lesch \cite{BrueningLesch1993}. In the context of complex varieties with isolated or more general singularities, in work of Cheeger, Goresky, and MacPherson \cite{CheegerGoreskyMacPherson}, Forn{\ae}ss, \O vrelid, and Vassiliadou \cite{FornaessOvrelidVassiliadou2005}, Grieser and Lesch \cite{GrieserLesch2002}, {\O}vrelid and Ruppenthal \cite{OvrelidRuppenthal2014},  Pardon and Stern \cite{PardonStern1991,PardonStern2001}, Ruppenthal \cite{Ruppenthal2014}, and Zucker \cite{Zucker1987}. A somewhat different category of problems concerns boundary value problems for regular elliptic complexes in which the boundary itself carries the singularities, two such examples being Shaw \cite{Shaw1983} where the boundary has conical singularities  and Tarkhanov \cite{Tarkhanov2006}, where the boundary has singularities of codimension $2$.

The detailed analysis of the nature of domains of the operators of an elliptic complex on manifolds with conical analysis presented in this paper is part of our ongoing research program to develop a general theory for elliptic complexes on manifolds with more general singularities.

\section{Preliminaries}\label{IBC}

We will continue to use the notation and objects already presented in the introduction throughout the rest of the paper without further comment. 

Let
\begin{equation*}
A_{q}^{\star} : C_c^{\infty}(\open\M;E^{q+1}) \subset x^{-\gamma}L^2_b(\M;E^{q+1}) \to x^{-\gamma}L^2_b(\M;E^{q})
\end{equation*}
be the formal adjoint of $A_q$. The $A_q^{\star}$ belong to $x^{-1}\Diff^1_b(\M;E^{q+1},E^q)$ and they form a cone-elliptic complex analogous to \eqref{TheCComplex}:
\begin{equation*}
\cdots\from C_c^\infty(\open \M,E^{q-1})\xleftarrow{A_{q-1}^\star}C_c^\infty(\open \M,E^q)\xleftarrow{A_q^\star}C_c^\infty(\open \M,E^{q+1})\from \cdots
\end{equation*}
Setting $P_q=xA_q$ we see that $A_q^\star$ is given by the formula
\begin{equation}\label{FormulaForAdjoint}
A_q^\star=x^{-2\gamma}P_q^\star x^{2\gamma-1}
\end{equation}
where $P_q^\star$ is the formal adjoint of 
\begin{equation*}
P_q:C_c^\infty(\open\M;E^q)\subset  L^2_b(\M;E^q)\to L^2_b(\M;E^{q+1}).
\end{equation*}
Let $\DomStar^q_{\max}$ and $\DomStar^q_{\min}$ be the maximal and minimal domains of $A_{q-1}^{\star}$ in $x^{-\gamma}L^2_b$, respectively. The Hilbert space adjoints in $x^{-\gamma}L^2_b$ satisfy
\begin{alignat*}{2}
(A_{q,\max})^* &= A_{q,\min}^{\star}, \quad (A_{q,\min})^* &&= A_{q,\max}^{\star}, \\
(A^{\star}_{q,\max})^* &= A_{q,\min}, \quad (A^{\star}_{q,\min})^* &&= A_{q,\max}.
\end{alignat*}
Consequently, the adjoint pairing
\begin{equation}\label{adjointpairing}
\begin{gathered}{}
[\cdot,\cdot]_{A_q} : \Dom_{\max}^q \times \DomStar_{\max}^{q+1} \to \C, \\
[u,v]_{A_q} = \langle A_qu,v \rangle_{x^{-\gamma}L^2_b} - \langle u,A_q^{\star}v \rangle_{x^{-\gamma}L^2_b},
\end{gathered}
\end{equation}
descends to a nondegenerate sesquilinear pairing
$$
[\cdot,\cdot]_{A_q} : \bigl(\Dom_{\max}^q/\Dom_{\min}^q\bigr) \times \bigl(\DomStar_{\max}^{q+1}/\DomStar_{\min}^{q+1}\bigr) \to \C.
$$
Equivalently, let $\Sing^q$ be the orthogonal complement of $\Dom_{\min}^q$ in $\Dom_{\max}^q$ with respect to $\langle\cdot,\cdot\rangle_q$, and let $\SingStar^{q+1}$ be the corresponding space for $A_{q}^{\star}$. Then
\begin{equation}\label{ReducedPairing}\tag{\ref{adjointpairing}$'$}
[\cdot,\cdot]_{A_q} : \Sing^q \times \SingStar^{q+1} \to \C
\end{equation}
is nondegenerate.

\begin{proposition}\label{ComplementsDegreeq}
We have
\begin{equation*}
\Sing^{q}=\Dom_{\max}^{q}\cap \ker(\Id+A_q^\star A_q),\quad \SingStar^{q}=\DomStar_{\max}^{q}\cap \ker(\Id+A_{q-1} A_{q-1}^\star).
\end{equation*}
Consequently, 
\begin{enumerate}
\item $\Sing^q\subset\ker A_{q-1}^\star$ and $\SingStar^{q}\subset\ker A_q$, which imply $\Sing^q\subset\DomStar_{\max}^q$ and $\SingStar^q\subset\Dom_{\max}^q$, respectively;\label{bulletintersect}
\item $\Sing^q + \SingStar^q$ is contained in the kernel of the operator
\begin{equation}\label{1plusDelta}
\Id+\square_q : \Dom_{\max}(\Id + \square_q) \subset x^{-\gamma}L^2_b(\M;E^q) \to x^{-\gamma}L^2_b(\M;E^q),
\end{equation}
where
$$
\square_q = A_q^{\star}A_q + A_{q-1}A_{q-1}^{\star} \in x^{-2}\Diff^2_b(\M;E^q)
$$
is the formal Laplacian associated with the complex \eqref{TheCComplex} and the base Hilbert space $x^{-\gamma}L^2_b$ in every degree;
\item\label{bulletFiniteDim} $\dim\Sing^q < \infty$ and $\dim\SingStar^q < \infty$;
\item\label{bulletHoloGain} there exists $\eps > 0$ such that $\Sing^q+\SingStar^q\subset x^{-\gamma+\eps}H^{\infty}_b(\M;E^q)$;
\item $\Sing^q+\SingStar^q \subset \Dom_{\max}^q \cap \DomStar^q_{\max}$ is a direct sum;\label{bulletdirect}
\item $\Sing^q+\SingStar^q$ is orthogonal to $\Dom_{\min}^q \cap \DomStar^q_{\min}$ with respect to the inner product
$$
\langle u,v \rangle_{\Dom_{\max}^q \cap \DomStar^q_{\max}} = \langle u,v \rangle_{x^{-\gamma}L^2_b} + \langle A_qu,A_qv \rangle_{x^{-\gamma}L^2_b} + \langle A_{q-1}^{\star}u,A_{q-1}^{\star}v \rangle_{x^{-\gamma}L^2_b}
$$
on $\Dom_{\max}^q \cap \DomStar^q_{\max}$.\label{bulletortho}
\end{enumerate}
\end{proposition}
\begin{proof}
The identity
$$
\Sing^{q}=\Dom_{\max}^{q}\cap \ker(\Id+A_q^\star A_q)
$$
for the orthogonal complement of the minimal domain in the maximal domain with respect to the graph inner product is generally true for differential operators acting on sections of Hermitian vector bundles on arbitrary Riemannian manifolds. A proof of this fact is given in \cite[Lemma~4.2]{GiKrMe07} in the framework of cone operators (while the operator is assumed to be elliptic in that reference the proof does not make use of ellipticity in any way).

We proceed to prove the remaining statements. If $u\in \Sing^q$, then certainly $u\in x^{-\gamma}L^2_b(\M;E^q)$. Also, $u=-A_q^\star  A_q u$, so $A_{q-1}^\star u=0\in x^{-\gamma}L_b^2(\M;E^{q+1})$. Now
$$
\square_q u = A_q^{\star} A_q u + A_{q-1} A_{q-1}^{\star} u = A_q^{\star} A_q u = -u.
$$
Thus $u \in x^{-\gamma}L^2_b(\M;E^q)$ with $(\Id + \square_q) u = 0$, and $(\Id + \square_q) \in x^{-2}\Diff^2_b(\M;E^q)$ is $c$-elliptic because the complex \eqref{TheCComplex} is $c$-elliptic by assumption. Consequently, the operator \eqref{1plusDelta} is Fredholm. Because $\Sing^q$ is contained in the kernel of \eqref{1plusDelta} we get $\dim\Sing^q < \infty$, see Lesch~\cite[Proposition 1.3.16]{Le97}. Moreover, elliptic regularity for cone operators implies that there exists $\eps > 0$ such that the kernel of \eqref{1plusDelta} is contained in $x^{-\gamma+\eps}H^{\infty}_b(\M;E^q)$. The same reasoning applies to $\SingStar^q$.

We show next that the sum in \eqref{bulletdirect} is direct. Let $u\in \Sing^q\cap\SingStar^q$. Then $u=-A_q^\star  A_q u$ gives $u\in \ker A^\star_{q-1}$. This together with $u=-A_{q-1} A_{q-1}^\star u$ gives $u=0$, as desired.

Finally, \eqref{bulletortho} follows from \eqref{bulletintersect} and the definition of $\Sing^q$ and $\SingStar^q$.
\end{proof}

Suppose $\Dom^q\subset x^{-\gamma}L^2_b(\M;E^q)$ is such that $\Dom_{\min}^q \subset \Dom^q\subset \Dom_{\max}^q$. As pointed out already, $A_q$ with domain $\Dom^q$ is closed because of (\ref{bulletFiniteDim}) of Proposition~\ref{ComplementsDegreeq} since $A_{q,\min}$ is already closed. So the only condition on the spaces if they are to give an ideal boundary condition is the compatibility condition $A_q(\Dom^q)\subset \Dom^{q+1}$. (Recall that by default $\Dom^m=x^{-\gamma}L^2_b(\M;E^m)$.) 

\begin{proposition}\label{HilbertComplex}
Let $\Dom^k\subset  x^{-\gamma}L^2_b(\M;E^q)$, $k=0,\dotsc,m$, be a choice of ideal boundary condition for the complex \eqref{TheCComplex}. Then the cohomology spaces of
\begin{equation}\label{HilbertComplexProper}
0\to\Dom^0\xrightarrow{A_0} \Dom^1\xrightarrow{A_1}\cdots \xrightarrow{A_{m-2}}\Dom^{m-1}\xrightarrow{A_{m-1}}\Dom^m\to 0
\end{equation}
in every degree are finite-dimensional. In particular, for each $q$, the space $A_q(\Dom^q)$ is a closed subspace of $ x^{-\gamma}L^2_b(\M;E^{q+1})$. 
\end{proposition}

The proof is a direct application of standard Hodge theory (taking advantage of the $c$-ellipticity of the complex and the compactness of $\M$). In the terminology of Br\"uning and Lesch, \eqref{HilbertComplexProper} is a Hilbert complex, and the finite-dimensionality of its cohomology groups makes it a Fredholm complex \cite[pg.~90]{BrueningLesch1992}.

\medskip
Let $\Ring$ denote the ring of $C^{\infty}$-functions on $\M$ that are constant on $\Z$. The following proposition shows that the maximal and minimal domains are localizable by elements of $\Ring$.

\begin{proposition}
Both $\Dom_{\min}^q$ and $\Dom_{\max}^q$ are $\Ring$-modules. Moreover, if $u \in \Dom_{\max}^q$ and $f \in \Ring$ such that $f|_\Z = 0$, then $fu \in \Dom_{\min}^q$.
\end{proposition}
\begin{proof}
We first prove that $\Dom_{\max}^q$ is an $\Ring$-module. To this end, let $f \in \Ring$ and $u \in \Dom_{\max}^q$ be arbitrary. Clearly, $fu \in x^{-\gamma}L^2_b(\M;E^q)$. Since $A_q$ is a first order differential operator, 
\begin{equation}\label{PpalSymbol}
A_q(f u)=fA_q(u)-\im \csym(A_q)(df)(u).
\end{equation}
This formula holds over $\open \M$ with the standard principal symbol and arbitrary smooth $f$, but since $df\in C^\infty(\M;\cT^*\M)$ because $f\in \Ring$, it holds with the $c$-symbol. Because $A_q(u) \in x^{-\gamma}L^2_b(\M;E^{q+1})$ we therefore obtain that both terms on the right in \eqref{PpalSymbol} are in $x^{-\gamma}L^2_b(\M;E^{q+1})$. Consequently $fu \in \Dom_{\max}^q$. If, in addition, $f|_\Z = 0$, then $fu \in x^{1-\gamma}L^2_b(\M;E^q) \cap \Dom_{\max}^q$, and thus $fu \in \Dom_{\min}^q$ by Proposition~\ref{MaxAndVanishingIsMin} below. That $\Dom_{\min}^q$ is an $\Ring$-module now follows by continuity of multiplication by $f$ in the graph norm, and from the fact that it leaves $C_c^{\infty}(\open\M;E^q)$ invariant.
\end{proof}

\section{The spaces $\Sing^q$}\label{gSing}

Let $\pi_\wedge:\Z^\wedge\to \Z$ be the closed inward pointing normal bundle of $\Z$ in $\M$. The form $dx$ defines a function on $\Z^\wedge$ which we continue to denote by $x$; it allows us to identify $\Z^\wedge$ with $\Z\times \lbra 0,\infty\rpar$. For any vector bundle $E_\Z \to \Z$ and every $\sigma_0 \in \C$ let $\gSing_{\sigma_0}(\Z^\wedge;E_\Z) \subset C^{\infty}(\open\Z^\wedge;E_\Z^\wedge)$ be the space of sections of $E_\Z^\wedge=\pi_\wedge^*E_\Z$ over $\open \Z^\wedge$ of the form
\begin{equation*}
u =  \sum_{k=0}^Nc_k(z)\,x^{i\sigma_0}\!\log^k x
\end{equation*}
for arbitrary $N \in \ZN_0$ and $c_k \in C^{\infty}(\Z;E_\Z)$. 

Suppose that $E_\Z$ is the restriction to $\Z$ of a bundle $E$ on $\M$. With the aid of a tubular neighborhood map and connections (all implicit and kept fixed from now on), identify a neighborhood of $\Z$ in $\M$ with a neighborhood of $\Z$ in $\Z^\wedge$, and sections of $E$ on $\M$ supported near $\Z$ with sections of the pull-back of $E_\Z$ to $\Z^\wedge$ that are supported near $\Z$. Let $\omega \in C_c^{\infty}(\overline{\R}_+)$ be a cut-off function that is supported sufficiently close to $x=0$, so generalized sections in $\omega C^{-\infty}(\open\Z^\wedge;E_\Z)$ make sense as generalized sections of $E$ on $\M$ supported near the boundary.

Suppose that $F\to \M$ is another vector bundle. The following lemma is a standard fact in analysis on $b$-manifolds.

\begin{lemma}\label{RegularityOfSing}
Let $u\in \gSing_{\sigma_0}(\Z^\wedge;E_\Z)$. Then:
\begin{enumerate}
\item \label{RegularityOfSing1} If $\mu>\Im\sigma_0$, then  $P\omega u\in x^{-\mu}L^2_b(\M;F)$ for every $P\in\Diff^*_b(\M;E,F)$. In particular, in this case, $\omega \gSing_{\sigma_0}(\Z^\wedge;E_\Z)\subset x^{-\mu}H_b^\infty(\M;E)$. 
\item \label{RegularityOfSing2} If $\mu\leq \Im \sigma_0$ and $u\in \gSing_{\sigma_0}(\Z^\wedge;E_\Z)$ satisfies $\omega u\in x^{-\mu}L^2_b(\M;E)$ then $u=0$.
\item \label{RegularityOfSing3} Now suppose that $u_{\sigma_j} \in \gSing_{\sigma_j}(\Z^\wedge;E_\Z)$ with $\sigma_j \in \C$, $j=1,\ldots,N$, and let
$$
u = \sum\limits_{j=1}^Nu_{\sigma_j}.
$$
Then $\omega u \in x^{-\mu}L^2_b(\M;E)$ if and only if every $\omega u_{\sigma_j} \in x^{-\mu}L^2_b(\M;E)$.
\end{enumerate} 
\end{lemma}

\begin{lemma}\label{OnDomMinGenericA}
Suppose that $A\in x^{-1}\Diff^1_b(\M;E,F)$. The minimal domain of
\begin{equation}
A: C_c^\infty(\open\M;E) \subset x^{-\gamma}L^2_b(\M;E)\to x^{-\gamma}L^2_b(\M;F).
\end{equation}
contains $x^{-\gamma+1}H^1_b(\M;E)$.  
\end{lemma}

Lemma~\ref{OnDomMinGenericA} follows at once from the continuity of
$$
A : x^{-\gamma+1}H^1_b(\M;E) \to x^{-\gamma}L^2_b(\M;F),
$$
the continuity of the embedding $x^{-\gamma+1}H^1_b(\M;E) \hookrightarrow x^{-\gamma}L^2_b(\M;E)$, and the density of $C_c^{\infty}(\open\M;E)$ in $x^{-\gamma+1}H^1_b(\M;E)$.

\medskip

The operator $P=xA$ has a Taylor expansion of order $N$ at $\Z$,
\begin{equation}\label{TaylorOfA}
P=\sum_{j=0}^N x^j P^{(j)}+x^{N+1}\tilde P^{(N+1)}\quad\text{near } \Z
\end{equation}
where $P^{(j)}\nabla_{x\partial_x}=\nabla_{x\partial_x}P^{(j)}$ for each $j$. The $P^{(j)}$ are indistinctly viewed as operators in $\Diff^1_b(\M;E,F)$ near $\Z$ or in $\Diff^1_b(\Z^\wedge;E_\Z,F_\Z)$, while $\tilde P^{(N+1)}$ is an element of $\Diff^1_b(\M;E,F)$ near $\Z$. Observe that $P^{(j)}$ maps $\gSing_{\sigma_0}(\Z^\wedge;E_\Z)$ into $\gSing_{\sigma_0}(\Z^\wedge;F_\Z)$. We will write
\begin{equation*}
A^{(j)}=\frac{1}{x}P^{(j)},\quad \tilde A^{(N+1)}=\frac{1}{x} \tilde P^{(N+1)}.
\end{equation*}

We now return to the complex \eqref{TheCComplex}. Using these Taylor series in the identity $A_{q+1}A_q=0$ gives
\begin{equation*}
A_{q+1}^{(0)}A_q^{(0)}=0.
\end{equation*}
In particular, for every $\sigma_0$ there is a chain complex 
\begin{equation}\label{gSingComplex}\stepcounter{equation}\tag{{\theequation}$_q$}
\cdots\to 
\gSing_{\sigma_0-\im}(\Z^\wedge;E^{q-1}_\Z)\xrightarrow{A_{q-1}^{(0)}}
\gSing_{\sigma_0}(\Z^\wedge;E^q_\Z)\xrightarrow{A_q^{(0)}}
\gSing_{\sigma_0+\im}(\Z^\wedge;E^{q+1}_\Z)\to\cdots.
\end{equation}

Since $\Sing^q\subset \Dom_{\max}^q\cap \ker(\square_q+\Id)$, if $u\in \Sing^q$ then, by the elliptic theory of a single cone operator,
\begin{equation}\label{SingStructureOfE}
u=\omega \sum_{\mathclap{\substack{\sigma_0\in \spec_b(\square_q)\\ \gamma-1\leq \Im\sigma_0<\gamma}}}u_{\sigma_0} + u'
\end{equation}
for suitable $u_{\sigma_0}\in \gSing_{\sigma_0}(\Z^\wedge;E^q_\Z)$ and $u'\in \Dom_{\min}^q$. Indeed, since $u\in \ker(\square_q+\Id)\cap x^{-\gamma}L^2_b(\M;E^q)$, it has such an expansion with $u'\in x^{-\gamma+1}H^{\infty}_b(\M;E^q)$, but as already noted in Lemma~\ref{OnDomMinGenericA}, $x^{-\gamma+1}H^{\infty}_b(\M;E^q)\subset \Dom_{\min}^q$. This establishes, with some ambiguities, an upper bound on the nature of the elements of $\Sing^q$. 

The ambiguities arise from several sources. On the one hand, the boundary spectrum of $\square_q$ may be bigger than what is necessary to describe the singular structure of the elements of $\Sing^q$. 

On the other, the determination of the minimal domain of $A_q$ is a somewhat more involved issue compared with the case of a single elliptic operator. One source of problems is the following:

\begin{example}\label{DminInterference}
Let $\sigma_0\in \C$ with $\gamma-1<\Im \sigma_0<\gamma$ (whether $\sigma_0\in \spec_b(\square_q)$ or not does not matter). Let $w\in \gSing_{\sigma_0-\im}(\Z^\wedge;E^{q-1}_\Z)$. Then $\omega w\in \Dom_{\min}^{q-1}$, hence $u=A_{q-1}\omega w\in \Dom_{\min}^q$. However,
\begin{equation*}
A_{q-1}(\omega w)=\omega u_{\sigma_0}+u'
\end{equation*}
with $u'\in x^{-\gamma+1}H^\infty_b(\M;E^q)$ and $u_{\sigma_0} = A_{q-1}^{(0)}w \in \gSing_{\sigma_0}(\Z^\wedge;E^q_\Z)$. By Lemma~\ref{OnDomMinGenericA}, $u'\in \Dom_{\min}^q$. So $\omega u_{\sigma_0}\in \Dom_{\min}^q$. This is in contrast with the theory of a single elliptic cone operator, in which there cannot be nonzero elements $\omega u_{\sigma_0}$ as above that belong to the minimal domain. 
\end{example}

\begin{proposition}\label{EisClosed}
Suppose $u_{\sigma_0}\in \gSing_{\sigma_0}(\Z^\wedge;E_\Z^q)$ with $\gamma -1 \leq\Im\sigma_0<\gamma$. Then $\omega u_{\sigma_0}\in \Dom_{\max}^q$ if and only if  $A_q^{(0)}u_{\sigma_0}=0$.

More generally, let $u_{\sigma_j} \in \gSing_{\sigma_j}(\Z^\wedge;E_\Z^q)$ with $\gamma -1 \leq \Im\sigma_j<\gamma$, $j=1,\ldots,N$, and let $u = \sum_{j=1}^Nu_{\sigma_j}$. Then $\omega u \in \Dom_{\max}^q$ if and only if every $A_q^{(0)} u_{\sigma_j} = 0$.
\end{proposition}

\begin{proof}
We have
\begin{equation*}
A_q(\omega u_{\sigma_0})=\omega A_qu_{\sigma_0}-\im \csym(A_q)(d\omega)(u_{\sigma_0}),
\end{equation*}
cf.~\eqref{PpalSymbol}. Since $\csym(A_q)(d\omega)$ is a smooth homomorphism that vanishes near $\Z$ (because $\omega=1$ near $\Z$), $\csym(A_q)(d\omega)(u_{\sigma_0})\in x^\infty L^2_b(\M;E^{q+1})$, in particular,  
\begin{equation*}
\csym(A_q)(d\omega)(u_{\sigma_0})\in x^{-\gamma}L^2_b(\M;E^{q+1}).
\end{equation*}
Hence $\omega u_{\sigma_0}\in \Dom_{\max}^q$ is equivalent to $\omega A_qu_{\sigma_0}\in x^{-\gamma}L^2_b(\M;E^{q+1})$. Now expand
\begin{equation*}
A_q=A_q^{(0)}+\tilde P_q^{(1)}
\end{equation*}
near $\Z$. Since $\tilde P_q^{(1)}\in \Diff^1_b(\M;E^q,E^{q+1})$ near $\Z$, $\omega \tilde P_q^{(1)}u_{\sigma_0}\in x^{-\Im\sigma_0-\eps}L^2_b$ for every $\eps>0$, hence $\omega \tilde P_q^{(1)}u_{\sigma_0}\in x^{-\gamma}L^2_b(\M;E^{q+1})$. It follows that 
\begin{equation*}
\omega A_q u_{\sigma_0}\in x^{-\gamma}L^2_b(\M;E^{q+1}) \iff \omega A_q^{(0)}u_{\sigma_0}\in x^{-\gamma}L^2_b(\M;E^{q+1}).
\end{equation*}
Since $A_q^{(0)}u_{\sigma_0}\in \gSing_{\sigma_0+\im}(\Z^\wedge;E_\Z^{q+1})$ and $\Im\sigma_0+1\geq \gamma$, 
\begin{equation*}
\omega A_q u_{\sigma_0}\in x^{-\gamma}L^2_b(\M;E^{q+1})\iff A_q^{(0)}u_{\sigma_0}=0
\end{equation*}
by (\ref{RegularityOfSing2}) of Lemma~\ref{RegularityOfSing}.

The proof of the second statement pertaining to linear combinations of singular sections follows along the same lines and makes use of (\ref{RegularityOfSing3}) of Lemma~\ref{RegularityOfSing} in the last step.
\end{proof}

To narrow down our characterization of $\Sing^q$ we need some basic information about $\Dom_{\min}^q$:

\begin{lemma}\label{MaxAndVanishingIsMinAux}
$\Dom_{\max}^q\cap \Big(\bigcap\limits_{\eps > 0}x^{-\gamma+1-\eps}H^1_b(\M;E^q)\Big) \subset \Dom_{\min}^q.$
\end{lemma}

\begin{proof}
Let $u$ be an element of the left hand side of the inclusion and let $u_\nu=x^{1/\nu} u$, $\nu\in \ZN$. Then $u_\nu \in x^{-\gamma+1}H^1_b(\M;E^q)$, so $u_\nu \in \Dom_{\min}^q$. We have $u_\nu \to u$ as $\nu \to \infty$ in $x^{-\gamma+1-\eps}H^1_b(\M;E^q)$ for every $\eps > 0$. Consequently, $A_qu_\nu \to A_qu$ in $x^{-\gamma-\eps}L^2_b(\M;E^{q+1})$ as $\nu \to \infty$ by continuity of $A_q : x^{-\gamma+1-\eps}H^1_b(\M;E^q) \to x^{-\gamma-\eps}L^2_b(\M;E^{q+1})$. 

Now let  $v \in \SingStar^{q+1}$ be arbitrary. By (\ref{bulletHoloGain}) of Proposition~\ref{ComplementsDegreeq} there is $\eps > 0$ such that $\SingStar^{q+1} \subset x^{-\gamma+\eps}H^{\infty}_b(\M;E^{q+1})$, that is, $x^{-\eps}v \in x^{-\gamma}H^{\infty}_b(\M;E^{q+1})$ with such $\eps$. Therefore, since $x^{\eps}A_qu_\nu \to x^{\eps}A_qu$ in $x^{-\gamma}L^2_b(\M;E^{q+1})$,
\begin{equation*}
\langle x^{\eps}A_qu_\nu,x^{-\eps}v \rangle_{x^{-\gamma}L^2_b} \to \langle x^{\eps}A_qu,x^{-\eps}v \rangle_{x^{-\gamma}L^2_b}\quad \text{as }\nu \to \infty,
\end{equation*}
that is,
\begin{equation*}
\langle A_qu_\nu,v \rangle_{x^{-\gamma}L^2_b} \to  \langle A_qu,v \rangle_{x^{-\gamma}L^2_b}\quad \text{as }\nu\to \infty.
\end{equation*}
Furthermore,
$$
\langle u_\nu,A_q^{\star}v \rangle_{x^{-\gamma}L^2_b} \to \langle u,A_q^{\star}v \rangle_{x^{-\gamma}L^2_b}\quad \text{as }\nu\rightarrow \infty,
$$
which together with the above shows that
$$
[u_\nu,v]_{A_q} \rightarrow [u,v]_{A_q}\quad \text{as }\nu\to \infty.
$$
Since $u_\nu \in \Dom_{\min}^q$ we have $[u_\nu,v]_{A_q} = 0$ for all $\nu \in \ZN$. Consequently, $[u,v]_{A_q} = 0$. Because $v \in \SingStar^{q+1}$ is arbitrary we get that $[u,w]_{A_q} = 0$ for all $w \in \DomStar_{\max}^{q+1}$, and consequently $u \in \Dom_{\min}^q$ as claimed.
\end{proof}

We shall improve on this lemma in Proposition~\ref{MaxAndVanishingIsMin}.

\begin{corollary}\label{BottomInDomMin}
Suppose the numbers $\set {\sigma_j}_{j=1}^N$ lie in $\Im\sigma_0=\gamma-1$. If $u_{\sigma_j}\in \gSing_{\sigma_j}(\Z^\wedge;E^q_\Z)$ and $\omega \sum_j u_{\sigma_j}\in \Dom_{\max}^q$, then $\omega \sum_j u_{\sigma_j}\in \Dom_{\min}^q$.
\end{corollary}

For the proof observe that $\omega \sum_j u_{\sigma_j}\in \Dom_{\max}^q$ by hypothesis and 
\begin{equation*}
\omega \sum_j u_{\sigma_j}\in \bigcap_{\eps>0} x^{-\gamma+1-\eps}H^1_b(\M;E^q),
\end{equation*}
a consequence of $\Im\sigma\leq \gamma-1$, see (\ref{RegularityOfSing1}) of Lemma~\ref{RegularityOfSing}.

\medskip
Thus, by Proposition~\ref{EisClosed} and Corollary~\ref{BottomInDomMin}, in the sum in \eqref{SingStructureOfE} the elements with $\Im\sigma_0=\gamma-1$ can be omitted because they belong to $\Dom_{\min}^q$. More specifically, the sum runs over the set
\begin{equation}\label{SigmaQGamma}
\Sigma_q^\gamma=\spec_b^q(A)\cap \set{\sigma:\gamma-1<\Im\sigma<\gamma}
\end{equation}
where $\spec_b^q(A)$, see Definition~\ref{bSpec}, is the set of points $\sigma_0\in \C$  for which the cohomology of \eqref{gSingComplex} in degree $q$ is nonzero. 

\begin{corollary}\label{SingReprestentative}
Every element of $\Dom_{\max}^q/\Dom_{\min}^q$ has a representative of the form 
\begin{equation*}
u=\omega \sum_{\mathclap{\sigma_0\in \Sigma_q^\gamma}}u_{\sigma_0}, \quad u_{\sigma_0}\in \gSing_{\sigma_0}(\Z^\wedge;E^q_\Z),\ A_{q}^{(0)}u_{\sigma_0}=0.
\end{equation*}
\end{corollary}

In view of Example~\ref{DminInterference}, elements
\begin{equation*}
\sum_{\mathclap{\sigma_0\in \Sigma_q^\gamma}} A_{q-1}^{(0)} w_{\sigma_0}, \quad w_{\sigma_0}\in \gSing_{\sigma_0-\im}(\Z^\wedge;E^{q-1}_\Z),
\end{equation*}
should be omitted. This shows the relevancy of \eqref{gSingComplex} as a complex. What is needed to completely narrow down the nature of the elements of $\Sing^q$ is the converse of Example~\ref{DminInterference}:

\begin{proposition}\label{DminInSigmaIsExact}
Suppose the numbers $\set {\sigma_j}_{j=1}^N$ lie in $\gamma-1< \Im\sigma <\gamma$. Let
\begin{equation*}
u=\sum_{j=1}^N u_j,\quad u_j\in \gSing_{\sigma_j}(\Z^{\wedge};E^q_\Z)
\end{equation*}
be such that $\omega u\in \Dom_{\min}^q$. Then there are $w_j\in \gSing_{\sigma_j-\im}(\Z^\wedge;E^{q-1}_\Z)$ such that
\begin{equation*}
A_{q-1}^{(0)}w=u,\quad w=\sum_{j=1}^N w_j.
\end{equation*}
\end{proposition}

Proposition~\ref{DminInSigmaIsExact} will be proved in Section~\ref{IndicialComplex}, utilizing the indicial cohomology of the complex \eqref{TheCComplex}, as a corollary to Theorems~\ref{MainTheorem} and \ref{AdjointPairingProps}. Granted Proposition~\ref{DminInSigmaIsExact}, we have identified $\Sing^q$ for each $q$:

\begin{theorem}\label{Theorem.Abis} For each $q$, the space $\Sing^q$ is canonically isomorphic to the direct sum over $\sigma_0\in \Sigma_q^\gamma$ of the cohomology spaces in degree $q$ of the complex
\begin{equation}\tag{\ref{gSingComplex}}
\cdots\to
\gSing_{\sigma_0-\im}(\Z^\wedge;E^{q-1}_\Z)\xrightarrow{A_{q-1}^{(0)}}
\gSing_{\sigma_0}(\Z^\wedge;E^{q}_\Z)\xrightarrow{A_{q}^{(0)}}
\gSing_{\sigma_0+\im}(\Z^\wedge;E^{q+1}_\Z)\to\cdots
\end{equation}
with $\Sigma_{q}^\gamma$ given by \eqref{SigmaQGamma}. That is, if  $\H^q_{\sigma_0}(\Z;A)$ denotes the $q$-th cohomology space of this complex, then \begin{equation*}
\Sing^q\cong \bigoplus_{\mathclap{\sigma_0\in \Sigma_q^\gamma}}\H^q_{\sigma_0}(\Z;A)
\end{equation*}
canonically.
\end{theorem}

Observe that the relevant complex depends on $q$: the cohomology of the complex \eqref{gSingComplex} in degrees $q'$ other than $q$ need not be related to the space $\Sing^{q'}$. Finally note that the complex \eqref{gSingComplex} is defined for arbitrary $\sigma_0\in \C$.

The adjoint complex of \eqref{TheCComplex} gives the complex 
\begin{equation*}
\cdots\leftarrow 
\gSing_{\sigma_0^\star+\im}(\Z^\wedge;E^{q}_\Z)\xleftarrow{A_q^{\star(0)}}
\gSing_{\sigma_0^\star}(\Z^\wedge;E^{q+1}_\Z)\xleftarrow{A_{q+1}^{\star(0)}}
\gSing_{\sigma_0^\star-\im}(\Z^\wedge;E^{q+2}_\Z)\leftarrow\cdots
\end{equation*}
analogous to \eqref{gSingComplex}. Here $\sigma_0^\star=\overline{\sigma_0-\im(2\gamma-1)}$ is the reflection of $\sigma_0$ across the line $\Im\sigma = \gamma-\frac{1}{2}$. If $u_{\sigma_0}$ is $A_q^{(0)}$-closed and $v_{\sigma_0^\star}$ is $A_q^{\star(0)}$-closed, then $\omega u_{\sigma_0}\in \Dom_{\max}^q$ and $\omega v_{\sigma_0^\star}\in \DomStar_{\max}^q$, so the pairing \eqref{adjointpairing} is defined on these sections. If furthermore $\omega u_{\sigma_0}\in \Dom_{\min}^q$, then
\begin{equation}\label{MiniPairing}
[\omega u_{\sigma_0},\omega v_{\sigma_0^\star}]_{A_q}=0\quad \forall v_{\sigma_0^\star}\in \gSing_{\sigma_0^\star}(\Z^\wedge;E^{q+1}),\ A_{q}^{\star (0)}v_{\sigma_0^\star}=0.
\end{equation}

The proof of Proposition~\ref{DminInSigmaIsExact} is based on a strong converse of this observation. More precisely, we will see that if $A_q^{(0)}u_{\sigma_0}=0$ and \eqref{MiniPairing} holds, then in fact $u_{\sigma_0}=A_{q-1}^{(0)}w$ for some $w\in \gSing_{\sigma_0-\im}(\Z^\wedge;E^{q-1}_\Z)$ (so in particular $\omega u_{\sigma_0}\in \Dom_{\min}^q$ after all). This is a statement about nondegeneracy of the pairing on $\H^q_{\sigma_0}(\Z;A)$ and $\H^{q+1}_{\sigma_0^\star}(\Z;A^\star)$ that is induced by \eqref{adjointpairing}.

\section{Indicial complex and boundary spectrum}\label{IndicialComplex}

We let $P_q=xA_q$, so $P_q\in \Diff^1_b(\M;E^q,E^{q+1})$, and let 
\begin{equation*}
\bP_q\in \Diff^1(\Z;E^q_\Z,E^{q+1}_\Z)
\end{equation*}
be the operator defined by $P_q$ along $\Z$. The indicial family of $A_q$ is
\begin{equation*}
\sigma\mapsto \bPar{x^{-\im\sigma+1} A_qx^{\im \sigma}},
\end{equation*}
by definition the same as that of $P_q$. It also coincides with the indicial family of $A_q^{(0)}$.

Let $\Lambda_q=\csym(A_q)(dx)$, a smooth homomorphism $E^q_\Z\to E^{q+1}_\Z$. The ellipticity assumption on \eqref{TheCComplex} gives that the sequence 
\begin{equation*}
\cdots \to E^{q-1}_\Z\xrightarrow{\Lambda_{q-1}}E^q_\Z\xrightarrow{\ \Lambda_q\ }E^{q+1}_\Z\to\cdots
\end{equation*}
is exact, so its Laplacian is positive definite. 

\begin{lemma}\label{indicialcomplexproperty}
The indicial family of $A_q$ is
\begin{equation*}
\sigma\mapsto \A_q(\sigma)=\bP_q+\sigma\Lambda_q.
\end{equation*}
It satisfies $\A_{q+1}(\sigma+\im)\A_q(\sigma)=0$ for all $\sigma \in \C$. The indicial family of $A_q^\star$ is 
\begin{equation}\label{sigmaStar}
\sigma\mapsto \A_q^\star(\sigma)=\bP_q^\star + \overline{\sigma^\star}\Lambda_q^\star, \quad \sigma^\star=\overline{\sigma-\im (2\gamma-1)}
\end{equation}
where $\bP_q^\star$ is the formal adjoint of 
\begin{equation*}
\bP_q:C^\infty(\Z;E^q_\Z)\subset L^2(\Z;E^q_\Z) \to L^2(\Z;E^{q+1}_\Z).
\end{equation*}
The adjoint of $\Lambda_q$ is defined using the already fixed Hermitian structures on the vector bundles, and $L^2$ spaces are constructed using these and the density $(x\partial_x\contract \m_b)|_\Z$ on $\Z$. The adjoint indicial family satisfies $\A_{q-1}^\star(\sigma+\im)\A_q^\star(\sigma)=0$ for all $\sigma \in \C$.
\end{lemma}

\begin{proof}
Using \eqref{PpalSymbol} with $f=x^{\im\sigma}$ yields
\begin{equation*}
x^{-\im\sigma + 1}A_q(x^{\im \sigma}\phi) = xA_q(\phi) + \sigma \csym(A_q)(dx)(\phi),
\end{equation*}
which gives the formula for the indicial family upon restriction to the boundary. The identity
\begin{equation*}
0 = x^{-\im\sigma +2}A_{q+1}A_q(x^{\im\sigma}\phi) =x^{-\im\sigma +2}A_{q+1}x^{\im \sigma-1}x^{-\im \sigma+1}A_q(x^{\im\sigma}\phi)
\end{equation*}
gives by restriction to $\Z$ that
\begin{equation*}
(\bP_{q+1}+(\sigma+\im)\csym(A_{q+1})(dx))(\bP_q+\sigma\csym(A_q)(dx))(\phi) = 0
\end{equation*}
as claimed. 

Formula \eqref{FormulaForAdjoint} asserts that $A_q^\star=x^{-2\gamma}P_q^\star x^{2\gamma-1}$ (where $P_q^\star$ is the formal adjoint of $P_q$ in $L^2_b$). Since the indicial families $\P_q$ and $\P_q^\star$ of $P_q$ and $P_q^\star$ are related by $\P_q^\star(\sigma)=\P_q(\overline\sigma)^\star$, the indicial family of $A_q^\star$ is
\begin{align*}
\bPar{x^{-\im\sigma+1}A_q^\star x^{\im \sigma}} 
&= \bPar{x^{-\im(\sigma -\im (2\gamma-1))} P_q^\star x^{\im(\sigma -\im (2\gamma-1))}}\\
&=\P_q^\star(\sigma -\im (2\gamma-1))\\
&=\P_q(\overline{\sigma -\im (2\gamma-1)})^\star\\
&=\bP_q^\star+(\sigma-\im(2\gamma-1))\Lambda_q^\star.
\end{align*}
which in view of the definition of $\sigma^\star$ is the stated the formula for $\A_q^\star(\sigma)$.
\end{proof}

\begin{lemma}\label{IntegrandAdjPair}
Let $U\subset \C$ be open and $\widehat u:U\to C^\infty(\Z;E^q_\Z)$ be holomorphic. Let $U^\star=\set{\sigma^\star:\sigma\in U}$ and let $\widehat v:U^\star\to C^\infty(\Z;E^{q+1}_\Z)$ be also holomorphic. Then
\begin{equation*}
\langle \A_q(\sigma)\widehat u(\sigma),\widehat v(\sigma^\star)\rangle_{L^2(\Z;E^{q+1}_\Z)}= \langle \widehat u(\sigma),\A_q^\star (\sigma^\star)\widehat v(\sigma^\star)\rangle_{L^2(\Z;E^{q+1}_\Z)}.
\end{equation*}
\end{lemma}

\begin{proof}
With the formulas in the lemma we have
\begin{equation*}
\langle (\bP_q+\sigma\Lambda_q)\widehat u(\sigma),\widehat v(\sigma^\star)\rangle_{L^2(\Z;E^{q+1}_\Z)}=
\langle \widehat u(\sigma),(\bP_q^\star+\overline\sigma \Lambda_q^\star)\widehat v(\sigma^\star)\rangle_{L^2(\Z;E^{q}_\Z)}.
\end{equation*}
But $\overline \sigma=\sigma^\star-\im(2\gamma-1)$, so
\begin{equation*}
\bP_q^\star+\overline\sigma \Lambda_q^\star=\bP_q^\star+(\sigma^\star-\im(2\gamma-1)) \Lambda_q^\star=\A_q^\star(\sigma^\star).
\end{equation*}
\end{proof}

\begin{proposition}\label{IndicialSpacesExact}
For each $q$ consider the holomorphic family of complexes
\begin{multline}\label{indicialspaces}\stepcounter{equation}\tag{{\theequation}$_q$}
\cdots \to C^{\infty}(\Z;E_\Z^{q-1}) \xrightarrow{\A_{q-1}(\sigma-\im)} C^{\infty}(\Z;E_\Z^q) \xrightarrow{\A_q(\sigma)} C^{\infty}(\Z;E_\Z^{q+1})\\\xrightarrow{\A_{q+1}(\sigma+\im )} C^{\infty}(\Z;E_\Z^{q+2}) \to \cdots
\end{multline}
This is an elliptic complex for each $\sigma \in \C$. There exists a discrete set $\Sigma^q \subset \C$ such that this complex is exact in degree $q$ for all $\sigma \notin \Sigma^q$. The set $\Sigma^q$ is such that $\Sigma^q \cap \{\sigma: |\Im(\sigma)| \leq \beta\}$ is finite for every $\beta>0$.
\end{proposition}
\begin{proof}
The indicial families form a complex \eqref{indicialspaces} for each $\sigma \in \C$ by Lemma~\ref{indicialcomplexproperty}. Their ellipticity follows at once from the $c$-ellipticity of the complex \eqref{TheCComplex}. Recalling that $\A^\star(\sigma)$ is defined to be $\A(\sigma^\star)^\star$, define
$$
\square_q(\sigma) = \A_q^\star(\sigma) \A_q(\sigma) + \A_{q-1}(\sigma-\im)\A_{q-1}^\star(\sigma+\im).
$$
This is a holomorphic family of elliptic differential operators acting on sections of the bundle $E^q_\Z\to\Z$. When $\Im\sigma =\gamma-1/2$, $\square_q(\sigma)$ is the Laplacian associated with the complex \eqref{indicialspaces}. Ellipticity of the complex \eqref{TheCComplex} implies further that $\square_q(\sigma)$ is elliptic with real parameter $\Re(\sigma)$ uniformly for $\Im(\sigma)$ in compact sets. Indeed, $\square_q(\sigma)$ equals 
\begin{multline*}
\bP_q^* \bP_q+\bP_{q-1}\bP_{q-1}^*+(\bP_q^* \Lambda_q  + \Lambda_q^* \bP_q+\Lambda_{q-1} \bP_{q-1}^* +  \bP_{q-1}\Lambda_{q-1}^*) \sigma \\+(\Lambda_q^* \Lambda_q+\Lambda_{q-1}\Lambda_{q-1}^* )\sigma^2;
\end{multline*}
because of ellipticity the term linear in $\sigma$ is dominated by the sum of the other terms when $|\Re\sigma|$ is large and $\Im\sigma$ remains in a compact set. Consequently, for each $\beta>0$ there is $\alpha > 0$ such that $\square_q(\sigma)^{-1}$ exists whenever $|\Im(\sigma)| \leq \beta$ and $|\Re(\sigma)| \geq \alpha$. By analytic Fredholm theory,
$$
\G_q(\sigma) = \square_q(\sigma)^{-1} : C^{\infty}(\Z;E^q) \to C^{\infty}(\Z;E^q)
$$
is a finitely meromorphic function taking values in $\Psi^{-2}(\Z;E^q_\Z)$. The set of poles of $\G_q(\sigma)$ is discrete in $\C$, and only finitely many poles are located in each horizontal strip of finite width. We have
$$
\square_q(\sigma)\A_{q-1}(\sigma-\im)\A_{q-1}^\star(\sigma+\im) = \A_{q-1}(\sigma-\im)\A_{q-1}^\star(\sigma+\im)\square_q(\sigma),
$$
and consequently also
$$
\G_q(\sigma)\A_{q-1}(\sigma-\im)\A_{q-1}^\star(\sigma+\im) = \A_{q-1}(\sigma-\im)\A_{q-1}^\star(\sigma+\im)\G_q(\sigma)
$$
when $\sigma$ is not a pole of $\G_q(\sigma)$. For such $\sigma$ we get
\begin{align*}
u &= \G_q(\sigma)\square_q(\sigma)u \\
&= \G_q(\sigma) \A_q^\star(\sigma)\A_q(\sigma)u + \A_{q-1}(\sigma-\im)\A_{q-1}^\star(\sigma+\im)\G_q(\sigma)u
\end{align*}
for all $u \in C^{\infty}(\Z;E^q_\Z)$. Consequently, $\ker(\A_q(\sigma))=\rg(\A_{q-1}(\sigma-\im))$ whenever $\sigma$ is not a pole of $\G_q(\sigma)$. The proposition is proved.
\end{proof}

If $V$ is a Fr\'echet space we will write $\Mero_{\sigma_0}(V)$ for the space of germs at $\sigma_0$ of meromorphic $V$-valued functions, $\Hol_{\sigma_0}(V)$ for the subspace of holomorphic germs, and let $\ss_{\sigma_0}(f)$ denote the singular part of a germ $f\in \Mero_{\sigma_0}(V)$. Thus $\ss_{\sigma_0}(f)$ is meromorphic in $\C$ with pole only at $\sigma_0$. We will often simply write $\ss$ instead of $\ss_{\sigma_0}$ from now on; what is meant will be clear from the context. 

Returning to the complex \eqref{TheCComplex} and its indicial complex \eqref{indicialspaces}, observe that because $\A_q$ depends holomorphically on $\sigma$, it gives maps 
\begin{equation*}
\begin{gathered}
\Hol_{\sigma_0}(C^\infty(\Z;E_\Z^q))\to\Hol_{\sigma_0}(C^\infty(\Z;E_\Z^{q+1})),\\
\Mero_{\sigma_0}(C^\infty(\Z;E_\Z^q))\to\Mero_{\sigma_0}(C^\infty(\Z;E_\Z^{q+1}))
\end{gathered}
\end{equation*}
on germs, so it induces mappings
\begin{equation*}
\Mero_{\sigma_0}(C^\infty(\Z;E_\Z^q))/\Hol_{\sigma_0}(C^\infty(\Z;E_\Z^q))\to \Mero_{\sigma_0}(C^\infty(\Z;E_\Z^{q+1}))/\Hol_{\sigma_0}(C^\infty(\Z;E_\Z^{q+1}))
\end{equation*}
for every $\sigma_0$ and every $q$ which assemble into a complex. Equivalently, if $\hat u$ is meromorphic near $\sigma_0$ with pole at $\sigma_0$, then
\begin{equation*}
\ss_{\sigma_0} \A_q(\widehat u)=\ss_{\sigma_0} \A_q(\ss_{\sigma_0}\widehat u)
\end{equation*}
It follows that $\ss\A_{q+1}(\cdot+\im)\ss\A_q(\cdot)=0$, and we have a localized complex 
\begin{multline}\label{indicialgerms}\stepcounter{equation}\tag{{\theequation}$_q$}
\cdots  \to \ss\Mero_{\sigma_0}(C^{\infty}(\Z;E_\Z^{q-1})) \xrightarrow{\ss \A_{q-1}(\cdot - \im)} \ss\Mero_{\sigma_0}(C^{\infty}(\Z;E_\Z^{q})) \xrightarrow{\ss \A_{q}(\cdot)}\\
\ss\Mero_{\sigma_0}(C^{\infty}(\Z;E_\Z^{q+1}))\xrightarrow{\ss \A_{q}(\cdot+\im)} \ss\Mero_{\sigma_0}(C^{\infty}(\Z;E_\Z^{q+2})) \to \cdots.
\end{multline}
at each $\sigma_0$.

This complex is isomorphic to the complex \eqref{gSingComplex} via the Mellin transform.  For any vector bundle $F_\Z \to \Z$, $\sigma_0 \in \C$, and $u\in \gSing_{\sigma_0}(\Z^\wedge;F_\Z)$ the Mellin transform of $\omega(x) u(x,z)$,
\begin{equation*}
\Mellin u(z,\sigma)=\int_0^{\infty} x^{-i\sigma}\omega(x)u(x,z)\,\frac{dx}{x},
\end{equation*}
is holomorphic in $\Im\sigma>\Im\sigma_0$ with values in $C^\infty(\Z,F_\Z)$. It extends as a meromorphic function to all of $\C$ with pole at $\sigma_0$.  

If $u\in \gSing_{\sigma_0}(\Z^\wedge;F_\Z)$, we may view $\Mellin u$ as representing a germ at $\sigma_0$ and then let $\ss\Mellin u$ be its singular part (at $\sigma_0$). Since the class of $\Mellin u$ modulo $\Hol_{\sigma_0}(F)$ is independent of the choice of the cut-off function $\omega$, so is $\ss\Mellin u$ and we have a map
\begin{equation*}
\gSing_{\sigma_0}(\Z^\wedge;F_\Z)\to \ss\Mero_{\sigma_0}(C^\infty(\Z;F_\Z))
\end{equation*}
This map is bijective with inverse 
$$
(\ss\Mellin)^{-1} \widehat u = \frac{1}{2\pi}\ointr_{C} x^{i\sigma}\widehat u(\sigma)\,d\sigma,\quad \widehat u(\sigma) \in \ss\Mero_{\sigma_0}(C^{\infty}(\Z;F_\Z))
$$
where $C$ is any clockwise oriented circle centered at $\sigma_0$. 

Define, for arbitrary $\vartheta\in \C$,
\begin{equation}\label{TauOperator}
\tau_\vartheta:\ss\Mero_{\sigma_0}(V)\to \ss\Mero_{\sigma_0-\vartheta}(V), \quad (\tau_\vartheta u)(\sigma)=u(\sigma+\vartheta).
\end{equation}
For each $q$, $q'$ define 
\begin{equation*}
\Theta^q_{q'}:\gSing_{\sigma_0+\im(q'-q)}(\Z^\wedge;E_\Z^q)\to \ss\Mero_{\sigma_0}(C^\infty(\Z^\wedge;E_\Z^q))
\end{equation*}
by
\begin{equation*}
\Theta^q_{q'} u= \tau_{\im(q'-q)}\ss\Mellin u.
\end{equation*}
Evidently these maps are isomorphisms. A brief calculation will show they give a commutative diagram
\begin{equation*}
\begin{CD}
\cdots @>>> \gSing_{\sigma_0}(\Z^\wedge;E^{q'}_\Z) @>{A_{q'}^{(0)}}>>\gSing_{\sigma_0+\im}(\Z^\wedge;E^{q'+1}_\Z) @>>> \cdots\\
@. @V{\Theta^q_{q'}}VV @V{\Theta^q_{q'+1}}VV @. \\
\cdots  @>>> \ss\Mero_{\sigma_0}(C^\infty(\Z;E^{q'}_\Z)) @>{\A_{q'}(\cdot+(q'-q)\im)}>>\ss\Mero_{\sigma_0}(C^\infty(\Z;E^{q'+1}_\Z)) @>>> \cdots
\end{CD}
\end{equation*}
so they define a chain isomorphism from the complex \eqref{gSingComplex} to the complex \eqref{indicialgerms}. Thus the cohomology space $\H^q_{\sigma_0}(\Z;A)$ of the previous section is identified with that of the complex \eqref{indicialgerms} (in the specific degree $q$).

\begin{definition}\label{bSpec}
The boundary spectrum of the complex \eqref{TheCComplex} in degree $q$ is
\begin{equation}\stepcounter{equation}\tag{{\theequation}$_q$}
\spec^{q}_b(A)= \{\sigma_0 \in \C;\; \H^q_{\sigma_0}(\Z;\A) \neq 0\}.
\end{equation}
where $\H^q_{\sigma_0}(\Z;\A)$ is the cohomology in degree $q$ of \eqref{indicialgerms}.
\end{definition}

Observe here again that the relevant complex depends on $q$; for a different degree $q'$ the complex should be centered at $q'$, or else, with the complex above, a shift. 

\begin{theorem}\label{bSpecProperties}
The space $\H^q_{\sigma_0}(\Z;\A)$ is finite-dimensional for every $\sigma_0 \in \C$. Moreover, $\spec_b^q(A) \cap \set{\sigma : |\Im(\sigma)| \leq \beta}$ is finite for every $\beta > 0$.
\end{theorem}
\begin{proof}
Using Proposition~\ref{IndicialSpacesExact}, Proposition~\ref{EllipticFiniteCohomMfds} implies that $\H^q_{\sigma_0}(\Z;\A)$ is finite-dimensional for every $\sigma_0 \in \C$. Moreover, Proposition~\ref{CohomGermsCarriedByCohom} implies that $\spec^q_b(A)$ is a subset of the set $\Sigma^q$ from Proposition~\ref{IndicialSpacesExact}. The theorem is proved.
\end{proof}

We complete the passage from the singular complexes \eqref{gSingComplex} to the complexes \eqref{indicialgerms} by transferring the canonical pairing \eqref{ReducedPairing} to a pairing in indicial cohomology.

\begin{proposition}\label{PairingInCohomology}
Suppose $\widehat{\mathbf u}\in \H^q_{\sigma_0}(\Z,\A)$, $\widehat{\mathbf v}\in \H^{q+1}_{\sigma_0^\star}(\Z,\A^\star)$ are represented, respectively, by elements $\widehat u\in \ss\Mero_{\sigma_0}(C^{\infty}(\Z;E^q_\Z))$ and $\widehat v\in \ss\Mero_{\sigma^{\star}_0}(C^{\infty}(\Z;E^{q+1}_\Z))$. Then 
\begin{equation*}
\frac{1}{2\pi}\ointc_C \langle \A_q(\sigma)\widehat u(\sigma),\widehat v(\sigma^\star)\rangle_{L^2(\Z;E^{q+1}_\Z)}\,d\sigma
\end{equation*}
depends only on the classes $\widehat{\mathbf u}$ and $\widehat{\mathbf v}$, and therefore defines a sesquilinear pairing
\begin{equation}\label{pairinggamma}
\left[\cdot,\cdot\right]^q_{\sigma_0,\sigma_0^{\star}} : \H^q_{\sigma_0}(\Z;\A) \times \H^{q+1}_{\sigma_0^{\star}}(\Z;\A^{\star}) \to \C.
\end{equation}
The curve $C$ is the boundary of a small disc centered at $\sigma_0$, oriented counterclockwise.
\end{proposition}

Note that $\widehat v(\sigma^\star)$ is antimeromorphic with pole at $\sigma_0$.

\begin{proof}
Let $\widehat u \in \ss\Mero_{\sigma_0}(C^{\infty}(\Z;E_\Z^q))$ and $\widehat v \in \ss\Mero_{\sigma^{\star}_0}(C^{\infty}(\Z;E_\Z^{q+1}))$, and suppose that $\ss \A_q\widehat u$ and $\ss \A_q^\star\widehat v$ both vanish, that is, $\A_q \widehat u$ and $\A_q^\star\widehat v$ are holomorphic near $\sigma_0$ and $\sigma_0^\star$, respectively. Let
\begin{equation*}
\widehat\varphi \in \ss\Mero_{\sigma_0}(C^\infty(\Z;E^{q-1}_\Z)),  \quad
\widehat\psi \in \ss\Mero_{\sigma^\star_0}(C^\infty(\Z;E^{q+2}_\Z))
\end{equation*}
be arbitrary. We will write $\rr$ for the regular part of a meromorphic germ at $\sigma_0$ or $\sigma_0^\star$ according to the context, so $\Id=\ss+\rr$. Omitting the reference to $L^2$ spaces we have 
\begin{align*}
\big\langle \A_q(\sigma)\big(\widehat u(\sigma)& + \ss\A_{q-1}(\sigma-\im)\widehat \varphi(\sigma)\big), \widehat v(\sigma^\star) + \ss\A_{q+1}^\star(\sigma^\star-\im )\widehat \psi(\sigma^\star) \big\rangle\\
&=
\big\langle \A_q(\sigma)\big(\widehat u(\sigma) + \A_{q-1}(\sigma-\im)\widehat \varphi(\sigma)-\rr\A_{q-1}(\sigma-\im)\widehat \varphi(\sigma)\big),\\
&\qquad\qquad \widehat v(\sigma^\star) + \A_{q+1}^\star(\sigma^\star-\im)\widehat \psi(\sigma^\star)-\rr\A_{q+1}^\star(\sigma^\star-\im)\widehat \psi(\sigma^\star) \big\rangle\\
&=\big\langle \A_q(\sigma)\big(\widehat u(\sigma) -\rr\A_{q-1}(\sigma-\im)\widehat \varphi(\sigma)\big),\\
&\qquad\qquad \widehat v(\sigma^\star) + \A_{q+1}^\star(\sigma^\star-\im)\widehat \psi(\sigma^\star)-\rr\A_{q+1}^\star(\sigma^\star-\im)\widehat \psi(\sigma^\star) \big\rangle
\end{align*}
by virtue of Lemma~\ref{indicialcomplexproperty}. Continuing with the last expression, using now also Lemma~\ref{IntegrandAdjPair} we get
\begin{align*}
\big\langle \widehat u(\sigma)& - \rr\A_{q-1}(\sigma-\im)\widehat \varphi(\sigma),\\
&\qquad\quad \A_q^\star(\sigma^\star)\big(\widehat v(\sigma^\star) + \A_{q+1}^\star(\sigma^\star-\im)\widehat \psi(\sigma^\star)-\rr\A_{q+1}^\star(\sigma^\star-\im)\widehat \psi(\sigma^\star)\big)\big\rangle\\
&=\big\langle \widehat u(\sigma) -\rr\A_{q-1}(\sigma-\im)\widehat \varphi(\sigma),
\A_q^\star(\sigma)\big(\widehat v(\sigma^\star) -\rr\A_{q+1}^\star(\sigma^\star-\im)\widehat \psi(\sigma^\star)\big) \big\rangle\\
&=  \big\langle \widehat u(\sigma),
\A_q^\star(\sigma^\star)\big(\widehat v(\sigma^\star) -\rr\A_{q+1}^\star(\sigma^\star)\widehat \psi(\sigma^\star)\big) \big\rangle \\
&\qquad\quad- \big\langle \rr\A_{q-1}(\sigma-\im)\widehat \varphi(\sigma),
\A_q^\star(\sigma)\big(\widehat v(\sigma^\star) -\rr\A_{q+1}^\star(\sigma^\star-\im)\widehat \psi(\sigma^\star)\big)  \big\rangle\\
&=  \big\langle \A_q(\sigma)\widehat u(\sigma),
\widehat v(\sigma^\star) -\rr\A_{q+1}^\star(\sigma^\star)\widehat \psi(\sigma^\star) \big\rangle \\
&\qquad\quad- \big\langle \rr\A_{q-1}(\sigma-\im)\widehat \varphi(\sigma),
\A_q^\star(\sigma)\big(\widehat v(\sigma^\star) -\rr\A_{q+1}^\star(\sigma^\star-\im)\widehat \psi(\sigma^\star)\big)  \big\rangle\\
&=  \big\langle \A_q(\sigma)\widehat u(\sigma), \widehat v(\sigma^\star)\big\rangle -\big\langle \A_q(\sigma)\widehat u(\sigma),\rr\A_{q+1}^\star(\sigma^\star)\widehat \psi(\sigma^\star) \big\rangle\\
&\qquad\quad- \big\langle \rr\A_{q-1}(\sigma-\im)\widehat \varphi(\sigma),
\A_q^\star(\sigma)\big(\widehat v(\sigma^\star) -\rr\A_{q+1}^\star(\sigma^\star-\im)\widehat \psi(\sigma^\star)\big)  \big\rangle 
\end{align*}
The last expression is equal to $\big\langle \A_q(\sigma)\widehat u(\sigma), \widehat v(\sigma^\star)\big\rangle$ modulo a function which is holomorphic near $\sigma_0$, so the contour integration gives the result.
\end{proof}

\begin{theorem}\label{MainTheorem}
The pairing \eqref{pairinggamma} is non-singular.
\end{theorem}
\begin{proof}
The proof is a corollary of Theorem~\ref{NondegenerateMfds}, as follows. Let 
\begin{equation*}
\P_{q'}(\sigma)=\A_{q'}(\sigma+\im(q'-q)+\im(\gamma-\tfrac{1}{2})).
\end{equation*}
Then $P_{q'+1}\circ\P_{q'}=0$ by Lemma~\ref{indicialcomplexproperty}. The explicit formula for the $\A_{q'}(\sigma)$ shows that the principal symbols of the $\P_{q'}(\sigma)$ are independent of $\sigma$. The exactness of the symbol sequence \eqref{ComplexPrincSymbol} follows from the $c$-ellipticity of the complex \eqref{TheCComplex}.

We have
$$
\P_{q'}(\overline\sigma)^{\star} = \A_{q'}(\overline{\sigma}+\im(q'-q)+\im(\gamma-\tfrac{1}{2}))^{\star} = \A_{q'}^{\star}(\sigma - \im(q'-q) + \im(\gamma-\tfrac{1}{2}))
$$
with $\A_{q'}^{\star}$ defined according to \eqref{sigmaStar}. The left-hand side equals $\P_{q'}^{\star}(\sigma)$ according to the notation that is used in Section~\ref{Reduction}, see formula \eqref{Pqstardef}. The translation operator $\tau= \tau_{\im(\gamma-\frac{1}{2})}$, see \eqref{TauOperator}, induces isomorphisms of chain complexes
$$
\begin{CD}
\cdots\rightarrow \ss\Mero_{\sigma_0}(C^\infty(\Z;E^{q'}_\Z)) @>{\ss\A_{q'}(\cdot+\im(q'-q))}>>\ss\Mero_{\sigma_0}(C^\infty(\Z;E^{q'+1}_\Z)) \rightarrow\cdots \\
@V{\tau}VV @V{\tau}VV  \\
\cdots\rightarrow \ss\Mero_{\sigma_1}(C^\infty(\Z;E^{q'}_\Z)) @>{\ss\P_{q'}}>>\ss\Mero_{\sigma_1}(C^\infty(\Z;E^{q'+1}_\Z)) \rightarrow\cdots
\end{CD}
$$
and
$$
\begin{CD}
\cdots\leftarrow \ss\Mero_{\sigma^{\star}_0}(C^\infty(\Z;E^{q'}_\Z)) @<{\ss\A^{\star}_{q'}(\cdot-\im(q'-q))}<< \ss\Mero_{\sigma^{\star}_0}(C^\infty(\Z;E^{q'+1}_\Z)) \leftarrow\cdots \\
@V{\tau}VV @V{\tau}VV  \\
\cdots\leftarrow \ss\Mero_{\overline\sigma_1}(C^\infty(\Z;E^{q'}_\Z)) @<{\ss\P^{\star}_{q'}}<<\ss\Mero_{\overline\sigma_1}(C^\infty(\Z;E^{q'+1}_\Z)) \leftarrow\cdots
\end{CD}
$$
where $\sigma_1 = \sigma_0 - \im(\gamma-\tfrac{1}{2})$.

Let $\widehat u\in \ss\Mero_{\sigma_0}(C^{\infty}(\Z;E^q_\Z))$ and $\widehat v\in \ss\Mero_{\sigma^{\star}_0}(C^{\infty}(\Z;E^{q+1}_\Z))$. Then $\tau$ applied to the germ
$$
\sigma \mapsto \langle \A_q(\sigma)\widehat{u}(\sigma),\widehat{v}(\sigma^{\star}) \rangle_{L^2(\Z;E_{\Z}^{q+1})} \in \Mero_{\sigma_0}(\C)
$$
equals the germ
$$
\sigma \mapsto \langle \P_q(\sigma)\widehat{\tau u}(\sigma),\widehat{\tau v}(\overline{\sigma}) \rangle_{L^2(\Z;E_{\Z}^{q+1})} \in \Mero_{\sigma_1}(\C).
$$
Consequently, the map $\tau$ carries the pairing \eqref{pairinggamma} between cohomology of the complexes $\A$ and $\A^{\star}$ at $\sigma_0$ and $\sigma^{\star}_0$ over to the pairing \eqref{pairingmfds} between cohomology of the complexes $\P$ and $\P^{\star}$ at $\sigma_1$ and $\overline{\sigma}_1$, respectively. Finally, the finite-dimensionality of the cohomology spaces in the hypothesis of Theorem~\ref{NondegenerateMfds} is guaranteed by Theorem~\ref{bSpecProperties}.
\end{proof}

\begin{theorem}\label{AdjointPairingProps}
Let $\mathbf u\in \Dom_{\max}^q/\Dom_{\min}^q$ and $\mathbf v\in \DomStar^{q+1}_{\max}/\DomStar^{q+1}_{\min}$ be represented, respectively, by 
\begin{equation*}
u = \omega\sum_{\mathclap{\sigma_0\in \Sigma_q^\gamma}} u_{\sigma_0},\ u_{\sigma_0} \in \gSing_{\sigma_0}(\Z^\wedge;E^q_\Z) \text{ and }
v = \omega\sum_{\mathclap{\sigma'_0\in \Sigma_{q+1}^{\star \gamma}}} v_{\sigma'_0},\ v_{\sigma'_0} \in \gSing_{\sigma'_0}(\Z^\wedge;E^{q+1}_\Z)
\end{equation*}
with $A_q^{(0)}u_{\sigma_0}=0$, and $A_q^{\star(0)}v_{\sigma'_0}=0$ in accordance with Corollary~\ref{SingReprestentative}. Let $\widehat{\mathbf u}_{\sigma_0}\in \H^q_{\sigma_0}(\Z;\A)$ and $\widehat{\mathbf v}_{\sigma'_0}\in \H^{q+1}_{\sigma'_0}(\Z;\A^{\star})$ be the classes represented by $\ss\Mellin u_{\sigma_0}$ and $\ss\Mellin v_{\sigma'_0}$. Then the adjoint pairing
$$
[\cdot,\cdot]_{A_q} : \big(\Dom_{\max}^q/\Dom_{\min}^q\big) \times \bigl(\DomStar_{\max}^{q+1}/\DomStar_{\min}^{q+1}\bigr) \to \C
$$
with respect to the $x^{-\gamma}L^2_b$-inner products is given by
$$
[u,v]_{A_q} = \sum_{\mathclap{\sigma_0\in \Sigma_q^\gamma}} \big[\widehat{\mathbf u}_{\sigma_0},\widehat{\mathbf v}_{\sigma_0^\star}\big]^q_{\sigma_0,\sigma_0^{\star}}
$$
with the pairings in cohomology given by \eqref{pairinggamma}.
\end{theorem}

Consequently, the only possibly nontrivial pairings between local cohomology spaces in the critical strip $\gamma-1 < \Im(\sigma) < \gamma$ are related by reflection about the middle line.

\begin{proof}[Proof of Theorem~\ref{AdjointPairingProps}]
To prove the formula for the adjoint pairing, it is enough to consider
\begin{equation*}
u = \omega u_{\sigma_0}, \quad v = \omega v_{\sigma'_0}
\end{equation*}
for $\gamma-1 < \Im(\sigma_0),\Im(\sigma'_0) < \gamma$. Then
\begin{align*}
[u,v]_{A_q}=\langle A_q u,v\rangle_{x^{-\gamma}L^2_b}-\langle  u,A_q^\star v\rangle_{x^{-\gamma}L^2_b}=\langle A_q^{(0)} u,v\rangle_{x^{-\gamma}L^2_b}-\langle  u,A_q^{\star(0)} v\rangle_{x^{-\gamma}L^2_b}.
\end{align*}
Via Plancherel's theorem,
\begin{align*}
\big\langle A_q^{(0)}u,v \big\rangle_{x^{-\gamma}L^2_b(\M;E^{q+1})}
&=\big\langle x^\gamma A_q^{(0)}u,x^\gamma v \big\rangle_{L^2_b(\Z^\wedge;E^{q+1}_\Z)}\\
&=\frac{1}{2\pi}\int_\R \big\langle\Mellin(x^\gamma A_q^{(0)}u)(s),\Mellin (x^\gamma v)(s)\big\rangle_{L^2(\Z,E^{q+1}_\Z)}\,ds.
\end{align*}
Using $xA_q^{(0)}=P_q^{(0)}$ we get 
\begin{equation*}
\Mellin(x^\gamma A_q^{(0)}u)(s)=\Mellin(x^\gamma P_q^{(0)}u)(s-\im)=\A_q(s+\im(\gamma-1))\Mellin(u)(s+\im(\gamma-1))
\end{equation*}
while
\begin{equation*}
\Mellin(x^\gamma v)(s)=\Mellin(v)(s+\im \gamma)=\Mellin(v)\big((s+\im (\gamma-1))^\star\big)
\end{equation*}
so that
\begin{equation*}
\big\langle A_q^{(0)}u,v \big\rangle_{x^{-\gamma}L^2_b}=\frac{1}{2\pi}\int_{\Im\sigma=\gamma-1}\big\langle \A_q(\sigma)\Mellin(u)(\sigma),\Mellin(v)\big(\sigma^\star\big)\big\rangle_{L^2}\,d\sigma
\end{equation*}
(with some notational imprecision for the inner products), of course with the contour of integration oriented by $s\mapsto s+\im (\gamma-1)$ with increasing $s$. Similarly,
\begin{align*}
\langle  u,A_q^{\star(0)} v\rangle_{x^{-\gamma}L^2_b}
&=\frac{1}{2\pi}\int_\R\big\langle \Mellin (u)(s+\im\gamma),\A_q^\star(s+\im(\gamma-1))\Mellin(v)(s+\im(\gamma-1))\big\rangle_{L^2}\,d\sigma\\
&=\frac{1}{2\pi}\int_\R\big\langle \Mellin(u)(s+\im\gamma),\A_q^\star((s+\im\gamma)^\star)\Mellin(v)((s+\im\gamma)^\star)\big\rangle_{L^2}\,d\sigma\\
&=\frac{1}{2\pi}\int_{\Im\sigma=\gamma}\big\langle \Mellin(u)(\sigma),\A_q^\star(\sigma^\star)\Mellin(v)(\sigma^\star)\big\rangle_{L^2}\,d\sigma
\end{align*}
with the contour of integration analogously oriented. Using Lemma~\ref{IntegrandAdjPair} we obtain
\begin{equation*}
\langle  u,A_q^{\star(0)} v\rangle_{x^{-\gamma}L^2_b}=\frac{1}{2\pi}\int_{\Im\sigma=\gamma}\big\langle \A_q(\sigma)\Mellin (u)(\sigma),\Mellin(v)(\sigma^\star)\big\rangle_{L^2}\,d\sigma.
\end{equation*}
Consequently, since $\A_q\Mellin(u)$ is entire and $\langle \A_q(\sigma)\Mellin(u)(\sigma),\Mellin(v)(\sigma^\star)\big\rangle_{L^2}$ therefore has a single pole at $\sigma_1 = (\sigma'_0)^\star$,
\begin{equation*}
[u,v]_{A_q}=\frac{1}{2\pi}\ointc_{C_{\sigma_1}}\big\langle \A_q(\sigma)\Mellin(u)(\sigma),\Mellin(v)(\sigma^\star)\big\rangle_{L^2}\,d\sigma
\end{equation*}
where $C_{\sigma_1}$ is a small counterclockwise oriented circle centered at $\sigma_1$.
We also have
\begin{equation*}
[u,v]_{A_q}=\frac{1}{2\pi}\ointc_{C_{\sigma_1}}\big\langle \Mellin(u)(\sigma),\A_q^\star(\sigma^\star)\Mellin(v)(\sigma^\star)\big\rangle_{L^2}\,d\sigma
\end{equation*}
again by virtue of Lemma~\ref{IntegrandAdjPair}. But if $\sigma_0\ne \sigma_1$ then  $C_{\sigma_1}$ can be chosen so that $\sigma_0$ is not on $C_{\sigma_1}$ or the region it encloses, and then the integral above is null, since $\A_q^\star(\sigma^\star)\Mellin(v)(\sigma^\star)$ is entire (antiholomorphic). The theorem is proved.
\end{proof}

We are now able to complete the argument from the previous section by proving Proposition~\ref{DminInSigmaIsExact}.

\begin{proof}[Proof of Proposition~\ref{DminInSigmaIsExact}]
As stated in the proposition, suppose that $\set {\sigma_j}_{j=1}^N$ lie in $\gamma-1<\Im\sigma<\gamma$, and let
\begin{equation*}
u=\sum_{j=1}^N u_j,\quad u_j\in \gSing_{\sigma_j}(\Z^{\wedge};E^q_\Z)
\end{equation*}
be such that $\omega u\in \Dom_{\min}^q$. Consequently, $[u,v]_{A_q} = 0$ for all $v \in \DomStar^{q+1}_{\max}$. Fix any $j \in \{1,\ldots,N\}$. Let $v = \omega v_{\sigma_j^{\star}}$, where $v_{\sigma_j^{\star}} \in \gSing_{\sigma_j^{\star}}(\Z^{\wedge};E_{\Z}^{q+1})$ is arbitrary such that $A_q^{\star(0)}v_{\sigma_j^{\star}} = 0$. Then
$$
0 = [u,v]_{A_q} =  \big[\widehat{\mathbf u}_{j},\widehat{\mathbf v}_{\sigma_j^\star}\big]^q_{\sigma_j,\sigma_j^{\star}}
$$
by Theorem~\ref{AdjointPairingProps} with the pairing \eqref{pairinggamma} on cohomology on the right-hand side. Because the cohomology class $\widehat{\mathbf v}_{\sigma_j^{\star}} \in \H_{\sigma_j^{\star}}^{q+1}(\Z;\A^\star)$ is arbitrary, Theorem~\ref{MainTheorem} implies that $\widehat{\mathbf u}_{j} = \mathbf 0 \in \H^q_{\sigma_j}(\Z;\A)$. Thus there exists $w_j \in \gSing_{\sigma_j-\im}(\Z^{\wedge};E_{\Z}^{q-1})$ such that $A_{q-1}^{(0)}w_j = u_j$. The proposition is proved.
\end{proof}

\begin{corollary}
$\Sing^q$ is canonically isomorphic to $\bigoplus_{\sigma_0\in \Sigma_q^\gamma} \H^q_{\sigma_0}(\Z;\A)$, where the $\H^q_{\sigma_0}(\Z;\A)$ are the cohomology spaces of \eqref{indicialgerms}.
\end{corollary}

\section{Holomorphic families of complexes in finite-dimensional spaces}\label{FiniteDimCase}

The core of the proof of Theorem~\ref{MainTheorem} is a similar statement, but on holomorphic families of complexes on finite dimensional spaces. These complexes are the subject of this section.

Let $F_q$, $q \in {\mathbb Z}$, be finite-dimensional complex vector spaces equipped with inner products. Suppose that, for each $q \in {\mathbb Z}$, we have elements
\begin{equation*}
\P_q \in \Hol_{\sigma_0}(\LL(F_q,F_{q+1}))
\end{equation*}
that satisfy $\P_{q+1}\circ \P_{q} = 0$ as germs. We thus have a holomorphic family of finite-dimensional complexes
\begin{equation}\label{FinComplexSpaces}
\cdots \to F_{q-1} \xrightarrow{\P_{q-1}(\sigma)}  F_q \xrightarrow{\P_q(\sigma)} F_{q+1} \to \cdots
\end{equation}
when $\sigma$ is sufficiently close to $\sigma_0$. In the same way that \eqref{indicialspaces} yields \eqref{indicialgerms}, the above complex gives
\begin{equation}\label{FinComplexGerms}
\cdots \to \ss\Mero_{\sigma_0}(F_{q-1}) \xrightarrow{\ss \P_{q-1}} \ss\Mero_{\sigma_0}(F_q) \xrightarrow{\ss \P_q} \ss\Mero_{\sigma_0}(F_{q+1}) \to \cdots
\end{equation}
with cohomology spaces 
\begin{equation}\label{FinCohomologyGerms}
\H^q_{\sigma_0}(\P)= \ker\ss \P_q / \rg \ss \P_{q-1}.
\end{equation}

Using the inner products on the spaces $F_q$ we define $\P_q^{\star}(\sigma)$ to be the element
\begin{equation*}
\P_q^{\star}(\sigma)= \P_q(\overline{\sigma})^\star \in \Hol_{\overline{\sigma}_0}\big(\LL(F_{q+1},F_q)\big).
\end{equation*}
Then $\P_{q-1}^{\star}\circ \P_q^{\star} = 0$ as holomorphic germs at $\overline \sigma_0$, and we get adjoint complexes
\begin{equation*}
\cdots \leftarrow F_{q-1} \xleftarrow{\P^{\star}_{q-1}(\sigma)}  F_q \xleftarrow{\P^{\star}_q(\sigma)} F_{q+1} \leftarrow \cdots  
\end{equation*}
and
\begin{equation*}
\cdots \leftarrow \ss\Mero_{\overline{\sigma}_0}(F_{q-1}) \xleftarrow{\ss \P^{\star}_{q-1}} \ss\Mero_{\overline{\sigma}_0}(F_q) \xleftarrow{\ss \P^{\star}_q} \ss\Mero_{\overline{\sigma}_0}(F_{q+1}) \leftarrow \cdots 
\end{equation*}
as above; here $\ss$ means singular parts at $\overline \sigma_0$. We set
\begin{equation*}
\H^q_{\overline{\sigma}_0}(\P^{\star})= \ker\ss \P_{q-1}^{\star}/\rg\ss \P_q^{\star}
\end{equation*}
so that cohomology spaces in degree $q$ are indexed consistently with superscript $q$. 

For each $q$ define a pairing
\begin{equation}\label{pairing}
\begin{gathered}
\langle \cdot,\cdot \rangle_{\P,\P^{\star}} : \Mero_{\sigma_0}(F_q) \times \Mero_{\overline{\sigma}_0}(F_{q+1}) \to \C, \\
\langle u,v \rangle_{\P,\P^{\star}}= \frac{1}{2\pi}\ointc_{C} \langle \P_q(\sigma)u(\sigma),v(\overline{\sigma}) \rangle_{F_{q+1}}\,d\sigma,
\end{gathered}
\end{equation}
where $C$ is a sufficiently small counterclockwise oriented circle centered at $\sigma_0$ so that all representatives of germs in the formula are defined and holomorphic in a neighborhood of the disk bounded by $C$ (except at the pole at the center). Note that
$$
\frac{1}{2\pi}\ointc_{C} \langle \P_q(\sigma)u(\sigma),v(\overline{\sigma}) \rangle_{F_{q+1}}\,d\sigma = \frac{1}{2\pi}\ointc_{C} \langle u(\sigma),(\P_q^{\star}v)(\overline{\sigma}) \rangle_{F_{q}}\,d\sigma,
$$
and consequently
$$
\langle v,u \rangle_{\P^{\star},\P} = -\overline{\langle u,v \rangle}_{\P,\P^{\star}}
$$
for $u \in \Mero_{\sigma_0}(F_q)$ and $v \in \Mero_{\overline{\sigma}_0}(F_{q+1})$.

\begin{proposition}
The pairing \eqref{pairing} induces a sesquilinear pairing in cohomology,
\begin{equation}\label{pairingcohom}
\langle \cdot,\cdot \rangle_{\P,\P^{\star}} : \H^q_{\sigma_0}(\P) \times \H^{q+1}_{\overline{\sigma}_0}(\P^{\star}) \to \C.
\end{equation}
\end{proposition}

The proof is identical to that of Proposition~\ref{PairingInCohomology}.

\begin{theorem}\label{NondegenerateFin}
Suppose $\dim \H^q_{\sigma_0}(\P) < \infty$. If $\mathbf u \in \H^q_{\sigma_0}(\P)$ is such that
$$
\langle \mathbf u,\mathbf v \rangle_{\P,\P^\star} = 0\text{ for all }\mathbf v \in \H^{q+1}_{\overline{\sigma}_0}(\P^\star), 
$$
then $\mathbf u = 0$.
\end{theorem}

\begin{corollary}\label{NondegFinCor}
Suppose both $\dim \H^q_{\sigma_0}(\P) < \infty$ and $\dim \H^{q+1}_{\overline{\sigma}_0}(\P^\star) < \infty$. Then these dimensions are equal, and the pairing \eqref{pairingcohom} is nondegenerate.
\end{corollary}

We now focus on the proof of Theorem~\ref{NondegenerateFin}. We first observe that we may assume without loss of generality that $\sigma_0 = 0$. To see this, let
\begin{gather*}
\tau_{\vartheta} : \Mero_{\sigma_1}(F) \to \Mero_{\sigma_1-\vartheta}(F), \\
(\tau_{\vartheta}u)(\sigma)= u(\sigma+\vartheta)
\end{gather*}
for arbitrary $\sigma_1$, $\vartheta$. This is an isomorphism with inverse $\tau_{-\vartheta}$. The map $\tau_{\vartheta}$ defines an isomorphism of chain complexes
$$
\begin{CD}
\cdots @>>> \ss\Mero_{\sigma_0}(F_q) @>{\ss \P_q}>> \ss\Mero_{\sigma_0}(F_{q+1}) @>>> \cdots \\
@. @V{\tau_{\sigma_0}}VV @V{\tau_{\sigma_0}}VV @. \\
\cdots @>>> \ss\Mero_{0}(F_q) @>{\ss \tau_{\sigma_0}\P_{q}\tau_{-\sigma_0}}>> \ss\Mero_{0}(F_{q+1}) @>>> \cdots.
\end{CD}
$$
Note that
$$
\big(\tau_{\sigma_0}\P_q\tau_{-\sigma_0}\big)(\sigma)=  \P_q(\sigma+\sigma_0) \in \Hol_{0}(\LL(F_q,F_{q+1}))
$$
and that
$$
\bigl(\tau_{\sigma_0}\P_q\tau_{-\sigma_0}\bigr)^{\star}(\sigma) = \bigl(\tau_{\overline{\sigma}_0}\P_q^{\star}\tau_{-\overline{\sigma}_0}\bigr)(\sigma).
$$
For the pairing \eqref{pairing} we get
$$
\langle u,v \rangle_{\P,\P^\star} = \langle \tau_{\sigma_0}u,\tau_{\overline{\sigma}_0}v \rangle_{\tau_{\sigma_0}\P\tau_{-\sigma_0},\tau_{\overline \sigma_0}\P^\star \tau_{-\overline \sigma_0}}
$$
for $u \in \Mero_{\sigma_0}(F_q)$ and $v\in \Mero_{\overline{\sigma}_0}(F_{q+1})$.

We thus consider henceforth complexes \eqref{FinComplexSpaces} and \eqref{FinComplexGerms} with $\sigma_0 = 0$.

\smallskip
Each space $F_q$ has an orthogonal Hodge-Kodaira decomposition
\begin{equation*}
F_q =  N_q \oplus R_q^{\star} \oplus \\ R_q 
\end{equation*}
associated with the complex \eqref{FinComplexSpaces} at $\sigma = 0$, where
\begin{align*}
N_q = \ker(\P_q(0))\cap\ker(\P_{q-1}^{\star}(0)), \qquad
R_q^{\star} = \rg(\P_{q}^{\star}(0)), \qquad
R_q = \rg(\P_{q-1}(0)).
\end{align*}
Accordingly, decompose $\P_q(\sigma) : F_q \to F_{q+1}$ as
\begin{equation}\label{Pqmatrixdecomposition}
\P_q(\sigma) = \begin{bmatrix}
\P_{q,11}(\sigma) & \P_{q,12}(\sigma) & \P_{q,13}(\sigma) \\[2pt]
\P_{q,21}(\sigma) & \P_{q,22}(\sigma) & \P_{q,23}(\sigma) \\[2pt]
\P_{q,31}(\sigma) & \P_{q,32}(\sigma) & \P_{q,33}(\sigma)
\end{bmatrix} :
\begin{array}{c} N_q \\ \oplus \\ R_q^{\star} \\ \oplus \\ R_q \end{array} \to
\begin{array}{c} N_{q+1} \\ \oplus \\ R_{q+1}^{\star} \\ \oplus \\ R_{q+1} \end{array}.
\end{equation}
All components are holomorphic near zero, all but $\P_{q,32}(\sigma)$ vanish at $\sigma=0$, and $\P_{q,32}(\sigma):R_q^{\star} \to R_{q+1}$ is invertible for $\sigma$ near $0$. Define $\tilde{\P}_q(\sigma) : N_q \to N_{q+1}$ by
\begin{equation}\label{Ptilde}
\tilde{\P}_q(\sigma)= \P_{q,11}(\sigma) - \P_{q,12}(\sigma)\P_{q,32}^{-1}(\sigma) \P_{q,31}(\sigma).
\end{equation}
The property 
\begin{equation}\label{tildePqvanish}
\tilde{\P}_q(0)=0
\end{equation}
will be used later. Of course the spaces $N_q$ are isomorphic to the cohomology spaces of the complex \eqref{FinComplexSpaces} at $\sigma_0$.

\begin{proposition}\label{TildeComplex}
The maps $\tilde{\P}_q(\sigma)$ are holomorphic near $\sigma =0$ and determine complexes
\begin{equation*}
\cdots \to N_{q-1} \xrightarrow{\tilde{\P}_{q-1}(\sigma)} N_q \xrightarrow{\tilde{\P}_q(\sigma)} N_{q+1} \to \cdots
\end{equation*}
and 
\begin{equation}\label{NQuotientComplex}
\cdots \to \ss\Mero_{0}(N_{q-1}) \xrightarrow{\ss\tilde{\P}_{q-1}} \ss\Mero_{0}(N_q) \xrightarrow{\ss\tilde{\P}_q} \ss\Mero_{0}(N_{q+1}) \to \cdots.
\end{equation}
Furthermore, there are natural (germs of) holomorphic chain maps $\Phi:F\to N$ and $\Psi:N\to F$ such that  $\Phi\circ \Psi=\Id$ and $\Psi\circ\Phi$ is homotopic to the identity map. Consequently, the cohomology groups of the complex \eqref{FinComplexGerms} are isomorphic to those of \eqref{NQuotientComplex}:
\begin{equation*}
\pmb \Phi_{q}:\H^q_0(\P)  \to \H^q_0(\tilde \P)
\end{equation*}
is an isomorphism.
\end{proposition}

\begin{proof}
That the $\tilde{\P}_q(\sigma)$ are holomorphic near $\sigma =0$ is evident.  Taking advantage of the decomposition \eqref{Pqmatrixdecomposition} define $Q_q(\sigma):F_q\to F_{q-1}$ by
\begin{equation*}
Q_q(\sigma)=\begin{bmatrix}
0&0&0\\
0&0&\P_{q,32}^{-1}(\sigma)\\
0&0&0
\end{bmatrix}
\end{equation*}
in all degrees $q$ and let $\Jay_q(\sigma):F_q\to F_q$ be the map
\begin{equation*}
\Jay_q(\sigma)=I_q-\big(Q_{q+1}(\sigma)\P_q(\sigma)+\P_{q-1}(\sigma)Q_q(\sigma)\big)
\end{equation*}
with $I_q$ the identity map. Evidently the $\Jay_q(\sigma)$ are holomorphic and define a chain map $\Jay(\sigma):F\to F$ which is chain homotopic to $I$.

We have, henceforth omitting the argument $\sigma$ and the references to the degrees, 
\begin{equation*}
Q\circ \P=\begin{bmatrix}
0&0&0\\[2pt]
\P_{32}^{-1}\P_{31}&\Id&\P_{32}^{-1}\P_{33}\\[2pt]
0&0&0
\end{bmatrix},
\qquad \P\circ Q=\begin{bmatrix}
0&0&\P_{12}\P_{32}^{-1}\\[2pt]
0&0&\P_{22}\P_{32}^{-1}\\[2pt]
0&0&\Id
\end{bmatrix}.
\end{equation*}
The entry in position $23$ of $Q\P+\P Q$ is
\begin{equation*}
\P_{22}\P_{32}^{-1}+\P_{32}^{-1}\P_{33}=\P_{32}^{-1}(\P_{32}\P_{22}+\P_{33}\P_{32})\P_{32}^{-1}=-\P_{32}^{-1}\P_{31}\P_{12}\P_{32}^{-1}
\end{equation*}
on account that the maps $\P$ define a complex.  So
\begin{equation*}
\Jay=
\begin{bmatrix}
I&0&-\P_{12}\P_{32}^{-1}\\[2pt]
-\P_{32}^{-1}\P_{31}&0&\P_{32}^{-1}\P_{31}\P_{12}\P_{32}^{-1}\\[2pt]
0&0&0
\end{bmatrix}.
\end{equation*}
Define $\Phi:F\to N$ and $\Psi:N\to F$ by 
\begin{equation*}
\Phi=\begin{bmatrix} I&0&-\P_{12}\P_{32}^{-1}\end{bmatrix},\qquad \\ \Psi=\begin{bmatrix}I \\[2pt] -\P_{32}^{-1} \P_{31}\\[2pt]0 \end{bmatrix}.
\end{equation*}
Then $\Psi\circ \Phi=\Jay$ and $\Phi\circ \Psi=\Id$. 

We now show that $\tilde \P\circ \tilde \P=0$. We have 
\begin{equation*}
\P\circ (I-Q\circ \P)=\begin{bmatrix}
\P_{11}-\P_{12}\P_{32}^{-1}\P_{31}&0&\P_{13}-\P_{12}\P_{32}^{-1}\P_{33}\\[2pt]
\P_{21}-\P_{22}\P_{32}^{-1}\P_{31}&0&\P_{23}-\P_{22}\P_{32}^{-1}\P_{33}\\[2pt]
0&0&0
\end{bmatrix}.
\end{equation*}
Since $(\P\circ \Jay)\circ (\P\circ \Jay)=0$ on account that $\P\circ \Jay=\Jay\circ \P$, the structure of the matrix representing $\P\circ (I-Q\circ \P)$ gives 
\begin{equation*}
(\P_{i1}-\P_{i2}\P_{32}^{-1}\P_{31})(\P_{1j}-\P_{12}\P_{32}^{-1}\P_{3j})=0,\quad i=1,2,\ j=1,3.
\end{equation*}
In particular with $i=j=1$ we have $\tilde \P_{q+1}(\sigma)\circ \tilde \P_q(\sigma)=0$. 

We leave the verification that $\tilde \P\circ \Phi=\Phi\circ \P$ and $\P\circ\Psi=\Psi\circ \tilde \P$ to the reader.

Finally observe that all maps involved are holomorphic near $\sigma=0$. 
\end{proof}

The adjoint $\tilde{\P}_q^{\star}(\sigma)$, defined (as before) via $\tilde{\P}_q^{\star}(\sigma) = \tilde{\P}_q(\overline{\sigma})^\star$, can alternatively be obtained consistently by following the same steps as above with the adjoint $\P_q^{\star}(\sigma)$ in lieu of $\P_q(\sigma)$ as can be seen from
\begin{equation*}
\P^{\star}_q(\sigma) = \begin{bmatrix}
\P^{\star}_{q,11}(\sigma) & \P^{\star}_{q,31}(\sigma) & \P^{\star}_{q,21}(\sigma) \\[2pt]
\P^{\star}_{q,13}(\sigma) & \P^{\star}_{q,33}(\sigma) & \P^{\star}_{q,23}(\sigma) \\[2pt]
\P^{\star}_{q,12}(\sigma) & \P^{\star}_{q,32}(\sigma) & \P^{\star}_{q,22}(\sigma)
\end{bmatrix} :
\begin{array}{c} N_{q+1} \\ \oplus \\ R_{q+1} \\ \oplus \\ R_{q+1}^{\star} \end{array}
\to
\begin{array}{c} N_q \\ \oplus \\ R_q \\ \oplus \\ R_q^{\star} \end{array}.
\end{equation*}
The roles of $\Psi$ and $\Phi$ are played, in the case of $\P^\star$, by $\Phi^*:N\to F$ and $\Psi^\star:F\to N$. In particular, $\Psi^\star$ induces isomorphisms
\begin{equation*}
\pmb \Psi_q^\star:\H^q_0(\P^\star)  \to \H^q_0(\tilde \P^\star)
\end{equation*}
in all degrees. 

\begin{proposition}\label{Reduction1}
Let $\mathbf u\in \H^q_0(\P)$ and $\mathbf v\in \H^{q+1}_0(\P^\star)$. Then
\begin{equation*}
\langle \pmb \Phi_q \mathbf u,\pmb \Psi^\star_{q+1} \mathbf v \rangle_{\tilde \P,\tilde \P^\star} = \langle \mathbf u, \mathbf v \rangle_{\P,\P^\star}
\end{equation*}
\end{proposition}
\begin{proof}
Let $u\in \Mero_0(F_q)$ represent $\mathbf  u\in \H^q_0(\P)$, let $v\in \Mero_0(F_{q+1})$ represent $\mathbf v\in \H^{q+1}_0(\P^\star)$. We have
\begin{align*}
\langle \tilde \P_q(\sigma)\Phi_q(\sigma) &u(\sigma),\Psi_{q+1}^*(\overline \sigma)v(\overline \sigma)\rangle_{N_{q+1}}
=\langle \Psi_{q+1}(\sigma)\tilde \P_q(\sigma)\Phi_q(\sigma) u(\sigma),v(\overline \sigma)\rangle_{F_{q+1}}\\
&=\langle  \P_q(\sigma)\Psi_q(\sigma)\Phi_qu(\sigma),v(\overline \sigma)\rangle_{F_{q+1}}\\
&=\langle  \P_q(\sigma)(I-\P_{q-1}(\sigma) Q_q(\sigma)-Q_{q+1}(\sigma)\P_q(\sigma))u(\sigma),v(\overline \sigma)\rangle_{F_{q+1}}.
\end{align*}
The last expression is equal to 
\begin{equation*}
\langle  \P_q(\sigma) u(\sigma),v(\overline \sigma)\rangle_{F_{q+1}}-\langle  Q_{q+1}(\sigma) \P_q(\sigma) u(\sigma),\P_q^\star(\overline \sigma) v(\overline \sigma)\rangle_{F_q}.
\end{equation*}
Since $\langle  Q_{q+1} \P_q u,\P_q^\star v\rangle_{F_q}$ is itself also holomorphic, 
\begin{equation*}
\langle \tilde \P_q(\sigma)\Phi_q(\sigma) u(\sigma),\Psi_{q+1}^*(\sigma)v(\sigma)\rangle_{N_{q+1}}\equiv \langle  \P_q(\sigma) u(\sigma),v(\overline \sigma)\rangle_{F_{q+1}}\mod\Hol_0(\C). 
\end{equation*}
The contour integration completes the proof.
\end{proof}

Multiplication by $\sigma$ gives a generic map
\begin{equation*}
\varsigma : \ss\Mero_0(F) \to \ss\Mero_0(F),
\end{equation*}
for any finite-dimensional vector space $F$, namely
\begin{equation*}
\varsigma(u)(\sigma)= (\sigma\mapsto \ss(\sigma u(\sigma))).
\end{equation*}
And if $F$ and $F'$ are finite-dimensional vector spaces and $\P \in \Hol_0(\LL(F,F'))$, then
$$
\ss \P\circ \varsigma = \varsigma \circ \ss \P : \ss\Mero_0(F) \to \ss\Mero_0(F').
$$
Consequently $\varsigma$ induces maps
$$
\ker\ss \P_q \to \ker\ss \P_q, \quad
\rg\ss \P_{q-1} \to \rg\ss \P_{q-1},
$$
so $\varsigma$ determines maps
\begin{equation}\label{SonCohom}
\pmb \varsigma : \H^q_0(\P) \to \H^q_0(\P)
\end{equation}
for each $q$: if $\mathbf u\in \H^q_0(\P)$ is represented by $u$, then $\pmb \varsigma \mathbf u=\varsigma(u)+\ss \P_{q-1}(\Mero_0(F_{q-1}))$.

For any $u \in \ss\Mero_0(F)$ we have $\varsigma^N u = 0$ if $N$ is at least the order of the pole of $u$ at zero. Consequently, if $\dim \H^q_0(\P) < \infty$, the induced map \eqref{SonCohom} is nilpotent on $\H^q_0(\P)$, in which case $\rg \pmb\varsigma \subsetneq \H^q_0(\P)$ is a proper subspace unless $\H^q_0(\P) = \set{0}$.

\begin{lemma}\label{sigmareduct}
Suppose the $\P_q \in \Hol_0(\LL(F_q,F_{q+1}))$ all satisfy $\P_q(0) = 0$. Then $\P_q(\sigma) = \sigma\hat{\P}_q(\sigma)$ with $\hat{\P}_q \in \Hol_0(\LL(F_q,F_{q+1}))$, and $\hat{\P}_{q+1}(\sigma)\hat{\P}_q(\sigma) = 0$ for $\sigma$ near $0$. The map
\begin{equation}\label{Jmap}
\bjay : \H^q_0(\hat \P) \ni u + \rg \ss \hat \P_{q-1} \mapsto u + \rg \ss \P_{q-1} \in \H^q_0(\P)
\end{equation}
is well-defined and injective onto the range of the map \eqref{SonCohom}. We have
\begin{equation}\label{JSFormula}
\langle \bjay \mathbf u,\mathbf v \rangle_{\P,\P^\star} = \langle \mathbf u, \bjay^{-1}\pmb \varsigma\mathbf v \rangle_{\hat \P,\hat \P^\star}
\end{equation}
for all $\mathbf u \in \H^q_0(\P)$ and all $\mathbf v \in \H^{q+1}_0(\P^\star)$, where
$$
\bjay^{-1} : \rg\pmb\varsigma \subset \H^{q+1}_0(\P^\star) \to \H^{q+1}_0(\hat \P^\star)
$$
in \eqref{JSFormula}.
\end{lemma}
\begin{proof}
Clearly, whenever $\hat \P_q(\sigma)u(\sigma) \in \Hol_0(F_{q+1})$, then $\P_q(\sigma)u(\sigma) \in \Hol_0(F_{q+1})$. Further, if $v\in \ss\Mero_0(F_{q-1})$ then also $\frac1\sigma v\in \ss\Mero_0(F_{q-1})$, hence $\hat \P_{q-1}v=\P_{q-1}\frac{1}{\sigma} v$, so $\ss \hat \P_{q-1}v\in \rg \ss \P_{q-1}$. So \eqref{Jmap} is well-defined. 

If $u$ represents $\mathbf u\in \H^q_0(\hat \P)$ and $u=\ss \P_{q-1} v$ for some $v\in \Mero_0(F_{q-1})$, then $u=\ss \hat \P_{q-1} \sigma v$ so $\mathbf u=0$ and $\bjay$ is injective.

To see that $\rg \bjay \subset \rg\pmb\varsigma $, let $u\in \ss\Mero_0(F_q)$ represent $\mathbf u\in \H^q_0(\hat \P)$. Then $u$ also represents $\bjay \mathbf u\in \H^q_0(\P)$ according to \eqref{Jmap}. Since $\hat \P_q u$ is holomorphic,  $\sigma\hat \P_q \frac{1}{\sigma}u=\P_q \frac{1}{\sigma}u$ is holomorphic, that is, $\frac{1}{\sigma}u\in \ker \ss \P_q$. So $w=\frac{1}{\sigma} u$ represents an element $\mathbf w\in \H^q_0(\P)$. The class $\pmb\varsigma\mathbf w$ is represented by $\ss(\sigma w) = u$, and so $\bjay \mathbf u = \pmb\varsigma\mathbf w \in \rg \pmb\varsigma$.

Conversely, every element $\mathbf u\in \rg\pmb\varsigma$ has a representative of the form $\ss (\sigma u)$ for some $u\in \ss\Mero_0(F_q)$ such that $\P_q u$ is holomorphic at zero. Since $\P_q u = \hat \P_q\sigma u$ is holomorphic at zero, so is $\hat \P_q \ss(\sigma u)$, and thus $\ss(\sigma u)$ represents an element of $\H^q_0(\hat \P)$, evidently mapped by $\bjay$ to $\mathbf u$. Thus $\rg \bjay =  \rg \pmb \varsigma$.

Finally, we note that \eqref{JSFormula} follows at once from
$$
\langle \P_q(\sigma)u(\sigma),v(\overline{\sigma}) \rangle_{F_{q+1}} = \langle \hat{\P}_q(\sigma)u(\sigma),\overline{\sigma}v(\overline{\sigma}) \rangle_{F_{q+1}}
$$
and the definition of the pairing in \eqref{pairing}.
\end{proof}

We are now ready to prove the main result of this section.

\begin{proof}[Proof of Theorem~\ref{NondegenerateFin}]
Without loss of generality let $\sigma_0 = 0$. The proof of the theorem proceeds by induction with respect to $\dim \H^q_0(\P)$. Clearly, if that dimension is zero, there is nothing to prove. So assume that the theorem is proved for complexes such that $\dim \H^q_0 \leq k$ for some $k \in \ZN_0$. Suppose then that we are dealing with a complex $\P$ such that $\dim \H^q_0(\P) \leq k+1$. By way of \eqref{tildePqvanish}, Propositions~\ref{TildeComplex} and \ref{Reduction1} reduce the consideration to the case that $\P_{q'}(0) = 0$ for all maps $\P_{q'}(\sigma)$ in the complex $\P$. Write $\P_{q'}(\sigma) = \sigma \hat{\P}_{q'}(\sigma)$ for all ${q'}$. Let $\mathbf u \in \H^q_0(\P)$ be such that
\begin{equation}\label{null}
\langle \mathbf u ,\mathbf v \rangle_{\P,\P^\star} = 0\text{ for all }\mathbf v \in \H^{q+1}_0(\P^\star).
\end{equation}
Let $v\in F_{q+1}$ be arbitrary. Then $\P_{q+1}^\star(\sigma) \frac{1}{\sigma}v$ is holomorphic at $\sigma=0$, so in \eqref{null} we may choose $\mathbf v$ to be the class of $\frac{1}{\sigma}v$. We get
$$
0 = \langle \mathbf u,\mathbf v \rangle_{\P,\P^\star} = \frac{1}{2\pi}\ointc_{C} \bigl\langle \P_q(\sigma)u(\sigma),\frac{v}{\overline{\sigma}}\bigr\rangle_{F_{q+1}}\,d\sigma = \im \langle h(0),v \rangle_{F_{q+1}},
$$
where $h(\sigma) = \P_q(\sigma)u(\sigma) \in \Hol_0(F_{q+1})$. Because $v \in F_{q+1}$ is arbitrary we conclude that $h(0) = 0$, and consequently $h(\sigma)=\sigma \hat{h}(\sigma)$ for some $\hat h\in \Hol_0(F_q)$. Since also $h(\sigma)=\sigma\hat \P_q(\sigma) u(\sigma)$, $\hat \P_q u$ is holomorphic near $0$. This implies that $\mathbf u$ is in the range of the map \eqref{Jmap} defined in Lemma~\ref{sigmareduct}: $\mathbf u = \bjay \tilde{\mathbf u}$ with $\tilde{\mathbf u} \in \H^q_0(\hat \P)$. Since $
\H^q_0(\hat \P)$ is isomorphic to $\rg\pmb\varsigma$ and $\pmb\varsigma$ is nilpotent, $\dim \H^q_0(\hat \P)\leq k$. Moreover, by \eqref{JSFormula}, we have
$$
\langle \tilde{\mathbf u},\tilde{\mathbf v} \rangle_{\P,\hat \P^\star} = 0
$$
for all $\tilde{\mathbf v} \in \H^{q+1}_0(\hat \P^\star)$. By the inductive hypothesis we now obtain that $\tilde{\mathbf u} = 0$, and consequently also $\mathbf u = \bjay\tilde{\mathbf u} = 0$. This finishes the induction and the proof of the theorem.
\end{proof}

We end the section with the following proposition giving a sufficient criterion for the finite-dimensionality of the cohomology space \eqref{FinCohomologyGerms} in terms of exactness of the original complex \eqref{FinComplexSpaces} for $\sigma$ near $\sigma_0$.

\begin{proposition}\label{FinEllipticFiniteCohom}
Suppose there exists $\eps > 0$ such that \eqref{FinComplexSpaces} is exact in degree $q$ for all $0 < |\sigma-\sigma_0| < \eps$. Then $\dim\H^q_{\sigma_0}(\P) < \infty$.
\end{proposition}
\begin{proof}
Let $\square_q(\sigma):F_q \to F_{q}$ be defined by 
\begin{equation*}
\square_q(\sigma)= \P_q^{\star}(\sigma - 2\im\Im(\sigma_0))\P_q(\sigma) + \P_{q-1}(\sigma)\P_{q-1}^{\star}(\sigma - 2\im\Im(\sigma_0)).
\end{equation*}
Then $\square_q \in \Hol_{\sigma_0}(\LL(F_q))$, and $\square_q(\sigma)$ is the 
Laplacian in degree $q$ of the complex \eqref{FinComplexSpaces} for $\sigma$ such that $|\sigma-\sigma_0|<\eps$ and $\Im(\sigma) = \Im(\sigma_0)$. By assumption, $\square_q(\sigma)^{-1} \in \LL(F_q)$ exists for such $\sigma$ if also $\sigma\ne \sigma_0$. Consequently we have
\begin{equation*}
\G_q(\sigma) = \square_q(\sigma)^{-1} \in \Mero_{\sigma_0}(\LL(F_q)).
\end{equation*}
Observe that
\begin{equation*}
\square_q(\sigma)\P_{q-1}(\sigma)\P_{q-1}^{\star}(\sigma - 2\im\Im(\sigma_0)) = \P_{q-1}(\sigma)\P_{q-1}^{\star}(\sigma - 2\im\Im(\sigma_0))\square_q(\sigma) \in \LL(F_q)
\end{equation*}
for all $\sigma$ near $\sigma_0$, and thus also
\begin{equation*}
\G_q(\sigma)\P_{q-1}(\sigma)\P_{q-1}^{\star}(\sigma - 2\im\Im(\sigma_0)) = \P_{q-1}(\sigma)\P_{q-1}^{\star}(\sigma - 2\im\Im(\sigma_0))\G_q(\sigma).
\end{equation*}
Now let $u \in \ss\Mero_{\sigma_0}(F_q)$ be arbitrary, and write
\begin{align*}
u(\sigma) &= \G_q(\sigma)\square_q(\sigma)u(\sigma) \\
&= \Bigl(\G_q(\sigma)\P_q^{\star}(\sigma - 2\im\Im(\sigma_0))\P_q(\sigma) + \G_q(\sigma)\P_{q-1}(\sigma)\P_{q-1}^{\star}(\sigma - 2\im\Im(\sigma_0))\Bigr)u(\sigma) \\
&= \G_q(\sigma)\P_q^{\star}(\sigma - 2\im\Im(\sigma_0))\P_q(\sigma)u(\sigma) + \P_{q-1}(\sigma)\P_{q-1}^{\star}(\sigma - 2\im\Im(\sigma_0))\G_q(\sigma)u(\sigma)\\
&=\ss\G_q(\sigma)\P_q^{\star}(\sigma - 2\im\Im(\sigma_0))\P_q(\sigma)u(\sigma) + \ss \P_{q-1}(\sigma)\P_{q-1}^{\star}(\sigma - 2\im\Im(\sigma_0))\G_q(\sigma)u(\sigma).
\end{align*}
Thus every $u \in \ss\Mero_{\sigma_0}(F_q)$ is of the form
\begin{equation}\label{FinCohomRep}
u(\sigma) \equiv \ss\G_q(\sigma)\P_q^{\star}(\sigma - 2\im\Im(\sigma_0))\P_q(\sigma)u(\sigma) \mod \rg\ss \P_{q-1}.
\end{equation}
Now, if $u$ is $\ss \P$-closed, then $\P_q(\sigma)u(\sigma)\in \Hol_{\sigma_0}(F_{q+1})$, and it follows that every cohomology class in $\H^q_{\sigma_0}(\P)$ has a representative in 
\begin{equation*}
{\mathcal U}_q = \set{\ss\G_q h : h \in \Hol_{\sigma_0}(F_q)} \subset \ss\Mero_{\sigma_0}(F_q).
\end{equation*}
This is a finite-dimensional space because the order of the pole at $\sigma_0$ of any of its elements is bounded by the order of the pole of $\G_q(\sigma)$ at $\sigma_0$. Thus $\dim\H^q_{\sigma_0}(\P) < \infty$ as desired.
\end{proof}

\section{Holomorphic families of elliptic complexes on closed manifolds}\label{Reduction}

Let $\Z$ be a closed manifold and let $E_{\Z}^q \to \Z$ be vector bundles for $q = 0,\dotsc,m$. Consider
$$
\P_q \in \Hol_{\sigma_0}(\Diff^1(\Z;E_{\Z}^q,E_{\Z}^{q+1}))
$$
such that $\P_{q+1}\circ \P_{q} = 0$ as germs. We assume that the principal symbol $\sym(\P_q)$ is independent of $\sigma$ for all $q$, and that the complex
\begin{equation}\label{ComplexPrincSymbol}
0 \to \pi^*E_{\Z}^0 \xrightarrow{\sym(\P_0)} \pi^*E_{\Z}^1 \to \cdots \to \pi^*E_{\Z}^{m-1} \xrightarrow{\sym(\P_{m-1})} \pi^*E_{\Z}^m \to 0
\end{equation}
is exact over $T^*\Z\minus 0$, where $\pi : T^*\Z \to \Z$ is the canonical projection. We thus obtain a holomorphic family of elliptic complexes
\begin{equation}\label{ComplexSpacesMfds}
\cdots \to C^{\infty}(\Z;E_{\Z}^{q}) \xrightarrow{\P_q(\sigma)} C^{\infty}(\Z;E_{\Z}^{q+1}) \to \cdots
\end{equation}
for $\sigma$ near $\sigma_0$. Using singular parts as in previous sections, we get an induced complex
\begin{equation}\label{ComplexGermsMfds}
\cdots \to \ss\Mero_{\sigma_0}(C^{\infty}(\Z;E_{\Z}^{q})) \xrightarrow{\ss \P_q} \ss\Mero_{\sigma_0}(C^{\infty}(\Z;E_{\Z}^{q+1})) \to \cdots.
\end{equation}
Let $\m$ a smooth positive density and suppose we are given Hermitian inner products on all the $E_{\Z}^q$. Using the resulting $L^2$-inner products we define
\begin{equation}\label{Pqstardef}
\P_q^{\star}(\sigma)= \P_q(\overline{\sigma})^\star \in \Hol_{\overline{\sigma}_0}(\Diff^1(\Z;E_{\Z}^{q+1},E_{\Z}^q)),
\end{equation}
where $\P_q(\overline{\sigma})^\star$ is the formal adjoint of $\P_q(\overline{\sigma}):C^{\infty}(\Z;E_{\Z}^{q}) \to C^{\infty}(\Z;E_{\Z}^{q+1})$. Again $\P_{q-1}^\star\circ \P_q^\star= 0$ as germs, and we get adjoint complexes
\begin{gather}
\cdots \leftarrow C^{\infty}(\Z;E_{\Z}^{q}) \xleftarrow{\P^{\star}_q(\sigma)} C^{\infty}(\Z;E_{\Z}^{q+1}) \leftarrow \cdots \label{ComplexSpacesAdjMfds} \\
\cdots \leftarrow \ss\Mero_{\overline \sigma _0}(C^{\infty}(\Z;E_{\Z}^{q})) \xleftarrow{\ss \P^ \star_q} \ss\Mero_{\overline \sigma_0}(C^{\infty}(\Z;E_{\Z}^{q+1})) \leftarrow \cdots. \label{ComplexGermsAdjMfds}
\end{gather}
The complex \eqref{ComplexSpacesAdjMfds} is elliptic for each $\sigma$. We extend all complexes trivially to all degrees $q \in {\mathbb Z}$. Let
\begin{equation*}
\H^q_{\sigma_0}(\Z;\P)= \ker \ss \P_q / \rg \ss \P_{q-1},\quad 
\H^q_{\overline{\sigma}_0}(\Z;\P^\star)= \ker \ss \P_{q-1}^ \star/\rg \ss \P_q^\star
\end{equation*}
be the cohomology spaces of \eqref{ComplexGermsMfds} and \eqref{ComplexGermsAdjMfds} in degree $q$, respectively. The analogue of Proposition~\ref{FinEllipticFiniteCohom} holds.

\begin{proposition}\label{EllipticFiniteCohomMfds}
Suppose there exists some $\eps > 0$ such that \eqref{ComplexSpacesMfds} is exact in degree $q$ for all $0 < |\sigma-\sigma_0| < \eps$. Then $\dim \H^q_{\sigma_0}(\Z;\P) < \infty$.
\end{proposition}
\begin{proof}
The proof of this proposition is an adaptation of that of Proposition~\ref{FinEllipticFiniteCohom}. Define $\square_q(\sigma) : C^{\infty}(\Z;E_{\Z}^q) \to C^{\infty}(\Z;E_{\Z}^q)$ by 
\begin{equation*}
\square_q(\sigma)= \P_q^{\star}(\sigma - 2\im\Im(\sigma_0))\P_q(\sigma) + \P_{q-1}(\sigma)\P_{q-1}^{\star}(\sigma - 2\im\Im(\sigma_0)).
\end{equation*}
Then $\square_q\in \Hol_{\sigma_0}(\Diff^2(\Z;E_{\Z}^q))$ is a holomorphic family of elliptic differential operators defined for $\sigma$ near $\sigma_0$ that coincides with the Laplacian in degree $q$ of the complex \eqref{ComplexSpacesMfds} for $\Im(\sigma) = \Im(\sigma_0)$. By assumption, $\square_q(\sigma)$ is invertible for $0 < |\sigma-\sigma_0| < \eps$ when $\Im(\sigma) = \Im(\sigma_0)$, and from analytic Fredholm theory we obtain that $\G_q(\sigma) = \square_q(\sigma)^{-1}$ is a meromorphic family of pseudodifferential operators; the Laurent coefficients of the singular part of $\G_q(\sigma)$ at $\sigma_0$,
\begin{equation*}
\ss\G_q=\sum_{\ell=-L}^{-1} \G_{q,_\ell}(\sigma-\sigma_0)^\ell,
\end{equation*}
are smoothing pseudodifferential operators of finite rank. 

The analogue of \eqref{FinCohomRep} holds also here for the same reasons, giving
\begin{equation*}
u(\sigma) \equiv \ss\G_q(\sigma)\P_q^{\star}(\sigma - 2\im\Im(\sigma_0)) \P_q(\sigma)u(\sigma) \mod \rg\ss \P_{q-1}
\end{equation*}
for any $u\in \ss\Mero_{\sigma_0}(C^\infty(\Z;E_{\Z}^q))$ and leading to the same conclusion as in the proof of Proposition~\ref{FinEllipticFiniteCohom}: every cohomology class in $\H^q_{\sigma_0}(\Z;\P)$ has a representative in
\begin{equation*}
\mathcal U_q = \set{\ss\G_q h : h \in \Hol_{\sigma_0}(C^\infty(\Z;E_{\Z}^q))}.
\end{equation*}
The space 
\begin{equation*}
V = \rg \G_{q,-L} + \dots + \rg \G_{q,-1} \subset C^{\infty}(\Z;E_{\Z}^{q})
\end{equation*}
is finite-dimensional because the operators $\G_{q,\ell}$ have finite rank. Since
\begin{equation*}
\mathcal V_q= \Set{\sum_{\ell=-L}^{-1} v_\ell (\sigma-\sigma_0)^\ell:v_\ell\in V}
\end{equation*}
is therefore also finite-dimensional, so is $\mathcal U_q\subset \mathcal V_q$, hence also $\H^q_{\sigma_0}(\Z;\P)$.
\end{proof}

Analogously to \eqref{pairing}, we consider for each $q$ the pairing
\begin{equation}\label{pairingmfds}
\begin{gathered}
\langle \cdot,\cdot \rangle_{\P,\P^\star} : \Mero_{\sigma_0}(C^{\infty}(\Z;E_{\Z}^q)) \times \Mero_{\overline{\sigma}_0}(C^{\infty}(\Z;E_{\Z}^{q+1})) \to \C, \\
\langle u,v \rangle_{\P,\P^\star}= \frac{1}{2\pi}\ointc_{C} \langle \P_q(\sigma)u(\sigma),v(\overline{\sigma}) \rangle_{L^2(\Z;E_{\Z}^{q+1})}\,d\sigma,
\end{gathered}
\end{equation}
where $C$ is a sufficiently small counterclockwise oriented circle centered at $\sigma_0$ so that all germs are defined and holomorphic in a neighborhood of the disk bounded by $C$ (except at the pole at the center). As in Proposition~\ref{PairingInCohomology}, this pairing induces a sesquilinear pairing
\begin{equation}\label{pairingcohommfds}
\langle \cdot,\cdot \rangle_{\P,\P^\star} : \H^q_{\sigma_0}(\Z;\P) \times \H^{q+1}_{\overline{\sigma}_0}(\Z;\P^\star) \to \C
\end{equation}
in cohomology. 

\begin{theorem}\label{NondegenerateMfds}
Suppose $\dim \H^q_{\sigma_0}(\Z;\P) < \infty$. If $\mathbf u \in \H^q_{\sigma_0}(\Z;\P)$ is such that
$$
\langle \mathbf u,\mathbf v \rangle_{\P,\P^\star} = 0\text{ for all }\mathbf v \in \H^{q+1}_{\overline \sigma _0}(\Z;\P^\star), 
$$
then $\mathbf u = 0$. In particular, if both $\dim \H^q_{\sigma_0}(\Z;\P) < \infty$ and $\dim \H^{q+1}_{\overline{\sigma}_0}(\Z;\P^\star) < \infty$, then these dimensions are equal and the pairing \eqref{pairingcohommfds} is nondegenerate.
\end{theorem}

The proof follows by reduction to the finite-dimensional case via the analogues of Propositions~\ref{TildeComplex} and Proposition~\ref{Reduction1}. We focus on defining the operators needed in Proposition~\ref{TildeComplex}.

Let
\begin{equation}\label{Laplacian0}
\square_q(0) = \P_q^{\star}(0)\P_q(0) + \P_{q-1}(0)\P_{q-1}^{\star}(0) : C^{\infty}(\Z;E_{\Z}^q) \to C^{\infty}(\Z;E_{\Z}^q)
\end{equation}
be the Laplacian of the complex \eqref{ComplexSpacesMfds} in degree $q$ at $\sigma = 0$, let
\begin{equation}\label{CohomologyAt0}
N_q = \ker(\square_q(0) : C^{\infty}(\Z;E_{\Z}^q) \to C^{\infty}(\Z;E_{\Z}^{q})).
\end{equation}
The ellipticity of the complex gives that $N_q\subset C^\infty(\Z;E_{\Z}^q)$ is finite-dimensional. Let $\Pi_{N_q}:L^2(\Z;E_{\Z}^q)\to L^2(\Z;E_{\Z}^q)$ be the orthogonal projection on $N_q$; this is a smoothing operator. Let $G_q$ be the Green's operator of \eqref{Laplacian0}, so 
\begin{equation*}
\square_q(0)G_q=G_q\square_q(0)=\Id-\Pi_{N_q}\text{ and } G_q\Pi_{N_q} = 0.
\end{equation*}
Define
\begin{equation*}
\Pi_{R_q}=\P_{q-1}(0)\P_{q-1}^\star(0)G_q,\quad \Pi_{R_q^\star}=\P_q^\star(0)\P_q(0)G_q.
\end{equation*}
These are zeroth order pseudodifferential operators (since $G_q$ is a pseudodifferential operator of order $-2$), the orthogonal projections $L^2(\Z;E_{\Z}^q)\to L^2(\Z;E_{\Z}^q)$ onto the $L^2$-closure of the subspaces
\begin{align*}
R_q^{\star} &= \rg(\P_q^{\star}(0) : C^{\infty}(\Z;E_{\Z}^{q+1}) \to C^{\infty}(\Z;E_{\Z}^q)), \\
R_q &= \rg(\P_{q-1}(0) : C^{\infty}(\Z;E_{\Z}^{q-1}) \to C^{\infty}(\Z;E_{\Z}^q)).
\end{align*}
From $\Id=\Pi_{N_q}+\Pi_{R_q}+\Pi_{R_q^\star}$ and the fact that pseudodifferential operators preserve smoothness one obtains the smooth Hodge-Kodaira decomposition
\begin{equation*}
C^{\infty}(\Z;E_{\Z}^q) = N_q \oplus  R_q^{\star} \oplus  R_q,
\end{equation*}
an orthogonal decomposition of $C^{\infty}(\Z;E_{\Z}^q)$ with respect to the $L^2(\Z;E_{\Z}^q)$-inner product which is also topological.

With the obvious definitions,
\begin{equation*}
\P_q(\sigma) = \begin{bmatrix}
\P_{q,11}(\sigma) & \P_{q,12}(\sigma) & \P_{q,13}(\sigma) \\[2pt]
\P_{q,21}(\sigma) & \P_{q,22}(\sigma) & \P_{q,23}(\sigma) \\[2pt]
\P_{q,31}(\sigma) & \P_{q,32}(\sigma) & \P_{q,33}(\sigma)
\end{bmatrix} :
\begin{array}{c} N_q \\ \oplus \\ R_q^{\star} \\ \oplus \\ R_q \end{array} \to
\begin{array}{c} N_{q+1} \\ \oplus \\ R_{q+1}^{\star} \\ \oplus \\ R_{q+1} \end{array}.
\end{equation*}
The entries $\P_{q,ij}(\sigma)$ are (restrictions to the respective domains of) first order pseudodifferential operators. For instance, $\P_{q,32}(\sigma)=\Pi_{R_{q+1}} \P_q(\sigma)\big|_{R_q^\star}$. Note that $\P_{q,ij}$ is smoothing if $i=1$ or $j=1$.

We claim that $\P_{q,32}$ has a holomorphic inverse for $\sigma$ near $0$ which is the restriction to $R_{q+1}$ of a holomorphic family of pseudodifferential operators of order $-1$. Indeed, let
\begin{equation*}
\B_q(0)=\P_q^\star(0)G_{q+1}.
\end{equation*}
From the definitions, $\P_q(0)\B_q(0)=\Pi_{R_{q+1}}$. Since the principal symbol of $\P_q(\sigma)$ is independent of $\sigma$, $\P_q(0)-\P_q(\sigma)$ is of order zero for each $\sigma$, so $\sigma\mathcal R_{q+1}(\sigma)=(\P_q(0)-\P_q(\sigma))\B_q(0)$ is a holomorphic family of pseudodifferential operators of order $-1$ that indeed vanishes at $\sigma=0$. We have
\begin{align*}
\Pi_{R_{q+1}}\P_q(\sigma)\B_q(0)&=\Pi_{R_{q+1}} \P_q(0)\B_q(0)-\Pi_{R_{q+1}}(\P_q(0)-\P_q(\sigma))\B_q(0)\\
&=\Pi_{R_{q+1}}(\Id-\sigma\mathcal R_{q+1}(\sigma))
\end{align*}
on account that $\Pi_{R_{q+1}}$ is a projection. Thus 
\begin{equation*}
(\Id -\sigma\mathcal R_{q+1}(\sigma))^{-1}=\sum_{k=0}^\infty \sigma^k\mathcal R_{q+1}(\sigma)^k,
\end{equation*}
defined for small $|\sigma|$, is a holomorphic family of pseudodifferential operators of order zero. Let $\B_q(\sigma)=\B_q(0)(\Id -\sigma\mathcal R_{q+1}(\sigma))^{-1}$. This operator maps $C^\infty(\Z;E_{\Z}^{q+1})$ into $R_q^\star$ and satisfies
\begin{equation*}
\Pi_{R_{q+1}}\P_q(\sigma)\B_q(\sigma) =\Pi_{R_{q+1}}
\end{equation*}
for $\sigma$ near $0$. The restriction of this formula to $R_{q+1}$ shows that 
\begin{equation}\label{Pq32map}
\P_{q,32}(\sigma):R_q^\star\to R_{q+1}
\end{equation}
is invertible from the right with right inverse 
\begin{equation*}
\B_q(\sigma)\big|_{R_{q+1}}:R_{q+1}\to R_q^\star. 
\end{equation*}
Invertibility of \eqref{Pq32map} from the left and construction of a holomorphic left inverse for $\sigma$ near $0$ follows along the same lines.

We now observe that Proposition~\ref{TildeComplex} and its proof hold verbatim with this new definition; the spaces are of course the $N_q$ and the operators are defined by the formula \eqref{Ptilde}. The cohomology groups of \eqref{ComplexGermsMfds} are thus isomorphic to those of the complex \eqref{NQuotientComplex}. Furthermore, the statement and proof of Proposition~\ref{Reduction1} also hold in the present context, so Theorem~\ref{NondegenerateFin} and Corollary~\ref{NondegFinCor} yield the conclusion in Theorem~\ref{NondegenerateMfds}.

Finally, we point out one additional important consequence of the arguments used here:

\begin{proposition}\label{CohomGermsCarriedByCohom}
Let $H^q(\Z;\P(\sigma_0))$ be the cohomology space in degree $q$ of the complex \eqref{ComplexSpacesMfds} for $\sigma=\sigma_0$. If $H^q(\Z;\P(\sigma_0)) = \set{0}$, then ${\H}^q_{\sigma_0}(\Z;\P) = \set{0}$.
\end{proposition}

Indeed, the space $N_q$ in \eqref{CohomologyAt0} is isomorphic to $H^q(\Z;\P(\sigma_0))$. Trivially, if $N_q=\set{0}$, then the cohomology group of \eqref{NQuotientComplex} in degree $q$ vanishes, thus also ${\H}^q_{\sigma_0}(\Z;\P) = \set{0}$.

\section{On the regularity of $\Dom_{\min}$}\label{DminRegularity}

We have completely characterized the space $\Sing^q$. In connection with $\Dom_{\min}$, no explicit characterization should be expected in general. 

\begin{example}\label{DminInterferenceBis}
Elaborating on Example~\ref{DminInterference}, let $f\in C_c^\infty(\C\times \Z;E^{q-1}_\Z)$ be supported in $\set{\sigma\in\C:\Im\sigma<\gamma-1}$, with $\supp f\cap \set{(\sigma,z)\in\C\times\Z:\Im\sigma>\gamma-2}\ne \emptyset$. Let
\begin{equation*}
w=\omega \int_\C x^{\im \sigma} f(z,\sigma)\,d\lambda(\sigma),
\end{equation*}
where $\lambda$ denotes Lebesgue measure on $\R^2$. Then $w\in \Dom_{\min}^{q-1}$, so $u=A_{q-1}w\in \Dom_{\min}^q$. But the Mellin transform of $u$ does not have a holomorphic extension to $\Im\sigma>\gamma-1$. With the same hypotheses on support, but with Sobolev regularity $H^1$ we also see that the regularity of elements in $\Dom_{\min}^q$ is difficult to characterize.
\end{example}

\begin{proposition}
There exist cone pseudodifferential operators
$$
B_q \in x^1\Psi_b^{-1}(\M;E^{q+1},E^q)
$$
in the small calculus such that
\begin{equation}\label{parametrixdegreeq}
B_q A_q + A_{q-1}B _{q-1}  = \Id - R_q 
\end{equation}
with $R_q \in \Psi_b^{-\infty}(\M;E^q,E^q)$ for all $q$ and, furthermore, 
$$
B_{q-1} B_q \in x^2\Psi_b^{-\infty}(\M;E^{q+1},E^{q-1}).
$$
\end{proposition}

\begin{proof}
Let
$$
\square_q = A_q^{\star} A_q + A_{q-1} A_{q-1}^{\star} \in x^{-2}\Diff_b^2(\M;E^q)
$$
be the (formal) Laplacian of the complex in degree $q$. Since the complex is $c$-elliptic, the Laplacian $\square_q$ is a $c$-elliptic operator. Let then $G_q \in x^{2}\Psi_b^{-2}(\M;E^q,E^q)$ be such that
$$
G_q\square_q - \Id, \; \square_q G_q - \Id \in \Psi_b^{-\infty}(\M;E^q);
$$
so $G_q$ is an approximate Green's operator in the small calculus. The fact that the $A_q$ (and the $A_q^{\star}$) form a complex implies that
\begin{align*}
A_q\square_q &= \square_{q+1}  A_q, \\
A_q^{\star}\square_{q+1} &= \square_q A_q^{\star},
\end{align*}
and consequently
\begin{align*}
A_q G_q &= G_{q+1} A_q \textup{ modulo }x^{1}\Psi_b^{-\infty}(\M;E^q,E^{q+1}), \\
A_q^{\star} G_{q+1} &= G_{q} A_q^{\star} \textup{ modulo }x^{1}\Psi_b^{-\infty}(\M;E^{q+1},E^q).
\end{align*}
Now define $B_q = A_q^{\star}G_{q+1} \in x^{1}\Psi_b^{-1}(\M;E^{q+1},E^q)$. Then
\begin{align*}
B_{q-1} B_q &= G_{q-1} A_{q-1}^{\star}  G_q  A_q^{\star} \\
&\equiv G_{q-1} A_{q-1}^{\star} A_q^{\star} G_{q+1} \\
&\equiv 0
\end{align*}
for $A_{q-1}^{\star} A_q^{\star} = 0$, where $\equiv$ means equality modulo $x^{2}\Psi_b^{-\infty}(\M;E^{q+1},E^{q-1})$. Similarly,
\begin{align*}
B_q  A_q + A_{q-1}  B_{q-1} &= G_q A_q^{\star} A_q + A_{q-1} G_{q-1} A_{q-1}^{\star} \\
&\equiv  G_q A_q^{\star} A_q + G_q  A_{q-1}  A_{q-1}^{\star} \\
&\equiv G_q \square_q \equiv \Id,
\end{align*}
where in this calculation $\equiv$ means equality modulo $\Psi_b^{-\infty}(\M;E^{q},E^{q})$.
\end{proof}

The next proposition improves on Lemma~\ref{MaxAndVanishingIsMinAux} in that the a priori regularity requirements on $u$ in that lemma can be weakened substantially, and the conclusion still remains valid.

\begin{proposition}\label{MaxAndVanishingIsMin}
We have
$$
\Dom_{\max}^q\cap \Bigl(\bigcap_{\eps > 0}x^{1-\gamma-\eps}H^{-\infty}_b(\M;E^q)\Bigr) \subset \Dom_{\min}^q.
$$
\end{proposition}
\begin{proof}
Let
$$
u \in \Dom_{\max}^q\cap \Bigl(\bigcap_{\eps > 0}x^{1-\gamma-\eps}H^{-\infty}_b(\M;E^q)\Bigr)
$$
be arbitrary. Use \eqref{parametrixdegreeq} and write
$$
B_q(A_q u) + A_{q-1}(B_{q-1} u) = u - R_q u.
$$
We have
$$
B_q(A_q u) \in x^{1-\gamma}H^1_b(\M;E^q) \subset \Dom_{\min}^q
$$
because $A_q u \in x^{-\gamma}L^2_b(\M;E^{q+1})$. We also have
$$
B_{q-1}u \in x^{1-\gamma}H^1_b(\M;E^{q-1}) \subset \Dom^{q-1}_{\min},
$$
and consequently $A_{q-1}(B_{q-1}u) \in \Dom_{\min}^q$ in view of $A_{q-1} : \Dom^{q-1}_{\min} \to \Dom_{\min}^q$. Thus
$$
u_0= B_q(A_q u) + A_{q-1}(B_{q-1} u) \in \Dom_{\min}^q.
$$
Now
$$
R_qu \in \bigcap_{\eps > 0}x^{1-\gamma-\eps}H^{\infty}_b(\M;E^q)
$$
by the mapping properties of the operator $R_q$ and our assumption about $u$, and
$$
R_qu = u - u_0 \in \Dom_{\max}^q.
$$
From Lemma~\ref{MaxAndVanishingIsMinAux} we then get that $R_qu \in \Dom_{\min}^q$, and thus $u = u_0 + R_qu \in \Dom_{\min}^q$ as desired.
\end{proof}

Examples~\ref{DminInterference} and \ref{DminInterferenceBis} show that, unlike in the case of a single elliptic operator discussed in detail in \cite[Proposition 3.6]{GiMe01}, the opposite inclusion in Proposition~\ref{MaxAndVanishingIsMin} does not hold.

\begin{theorem}
Let $u \in \Dom_{\min}^q$. Then there exists $v \in \Dom^{q-1}_{\min}$ such that
\begin{equation*}
u - A_{q-1}v \in \bigcap_{\eps > 0}x^{1-\gamma-\eps}H^1_b(\M;E^q).
\end{equation*}
\end{theorem}

\begin{proof}
Recall from Proposition~\ref{HilbertComplex} that every $c$-elliptic Hilbert complex of cone operators is Fredholm. This is applied here to the relative complex, i.e., the complex where every operator $A_q$ acts on its minimal domain. Thus $\rg(A_{q-1,\min})$ is closed, and it follows that 
\begin{equation*}
\Dom_{\min}^q = \bigl[\ker(A_{q-1,\max}^\star)\cap\Dom_{\min}^q\bigr] \oplus \rg(A_{q-1,\min}).
\end{equation*}
Let $u \in \ker(A_{q-1,\max}^\star)\cap\Dom_{\min}^q$ be arbitrary. Then
\begin{equation*}
\square_q u = \bigl( A_q^{\star}A_q + A_{q-1}A_{q-1}^{\star} \bigr) u = A_q^{\star}A_q u \in x^{-\gamma-1}H^{-1}_b(\M;E^q)
\end{equation*}
because $A_q u \in x^{-\gamma}L^2_b(\M;E^{q+1})$. Consequently, $u \in x^{-\gamma}L^2_b(\M;E^q)$ is such that $\square_q u \in x^{-\gamma-1}H^{-1}_b(\M;E^q)$. Because $\square_q u\in x^{-2}\Diff^2_b(\M;E^q,E^q)$ is $c$-elliptic, elliptic regularity theory for cone operators implies that there exist
\begin{equation*}
u_c\in \bigcap_{\eps > 0}x^{-\gamma+1-\eps}H^1_b(\M;E^q),\ u_s=\sum_{j=1}^N u_{\sigma_j},\  u_{\sigma_j} \in \gSing_{\sigma_j}(\Z^{\wedge};E^q_\Z),\ \gamma-1 < \Im(\sigma_j) < \gamma,
\end{equation*}
such that 
$
u=u_c+\omega u_s.
$
 We get
\begin{equation*}
A_q(\omega u_s) = A_qu - A_qu_c \in \bigcap_{\eps > 0}x^{-\gamma-\eps}L^2_b(\M;E^q),
\end{equation*}
and therefore necessarily $A_q^{(0)}u_s = 0$ due to the location of the $\sigma_j$. But then $A_q(\omega u_s) \in x^{-\gamma}L^2_b(\M;E^q)$. So $\omega u_s\in \Dom_{\max}^q$, hence $u-\omega u_s\in \Dom_{\max}^q$, but then
\begin{equation*}
u_c = u - \omega u_s \in \Dom_{\max}^q \cap \Bigl(\bigcap_{\eps > 0}x^{-\gamma+1-\eps}H^1_b(\M;E^q)\Bigr) \subset \Dom_{\min}^q,
\end{equation*}
by Lemma~\ref{MaxAndVanishingIsMinAux}. This shows that $\omega u_s = u - u_c \in \Dom_{\min}^q$. By Proposition~\ref{DminInSigmaIsExact} there exists
\begin{equation*}
v_s = \sum_{j=1}^N v_{\sigma_j-i},\  v_{\sigma_j-i} \in \gSing_{\sigma_j-i}(\Z^{\wedge};E^{q-1}_\Z),
\end{equation*}
such that $A_{q-1}^{(0)}v_s = u_s$. Let $v=\omega v_s$. We have $v \in x^{-\gamma+1}H^{\infty}_b(\M;E^{q-1}) \subset \Dom^{q-1}_{\min}$, and $A_{q-1}v = \omega u_s - w$ for some $w \in x^{-\gamma+1}H^{\infty}_b(\M;E^q)$. In conclusion,
$$
u - A_{q-1}v = u - \omega u_s + w = u_c + w \in \bigcap_{\eps > 0}x^{1-\gamma-\eps}H^1_b(\M;E^q)
$$
as desired.
\end{proof}

\section{The variety of ideal boundary conditions}\label{Variety}

Let $\pmb \Dom=(\Dom^0,\dotsc,\Dom^m)$ be a list of subspaces $\Dom^q\subset \Dom_{\max}^q$ with $\Dom_{\min}^q\subset \Dom^q$, $q=0,\dots,m-1$ and $\Dom^m=x^{-\gamma}L^2_b(\M;E^m)$. Each $\Dom^q$ decomposes as $\Dom^q=D^q+\Dom_{\min}^q$ with a unique $D^q\subset \Sing^q$, thus giving a new list $\mathbf D=(D^0,\dots,D^{m-1})$. The space $D^m$ is null, so we omit it.

Let $\pi^q_{\max}:\Dom_{\max}^q\to \Dom_{\max}^q$ be the orthogonal projection on $\Sing^q$ according to the inner product \eqref{InnerProduct} and let $\pi_{\min}^q=\Id-\pi^q_{\max}$. Set 
\begin{equation}\label{afrak}
\a_q=\pi^{q+1}_{\max} A_q\big|_{\Sing^q}:\Sing^q\to \Sing^{q+1}
\end{equation}
The maps $\a_q$ form a complex. To see this, write $A_{q+1}A_q$ as
\begin{equation*}
\pi^{q+2}_{\max}A_{q+1}\pi^{q+1}_{\max}A_q+
\pi^{q+2}_{\max}A_{q+1}\pi_{\min}^{q+1}A_q+
\pi_{\min}^{q+2}A_{q+1}A_q
\end{equation*}
and observe that $\pi_{\min}^{q+2}A_{q+1}A_q = 0$ and $\pi^{q+2}_{\max}A_{q+1}\pi_{\min}^{q+1}A_q=0$, so
\begin{equation*}
\pi^{q+2}_{\max}A_{q+1}\pi^{q+1}_{\max}A_q=0.
\end{equation*}

The property that $\pmb\Dom$ is an ideal boundary condition is reflected in the analogous property for the sequence $\mathbf D$:
\begin{lemma}\label{IdealBdy}
$\pmb \Dom$ is an ideal boundary condition if and only if $\a_q D^q\subset D^{q+1}$ for each $q$.
\end{lemma}

Indeed, since $A_q$ maps $\Dom_{\min}^q$ into $\Dom_{\min}^{q+1}$, the condition $A_q(\Dom^q)\subset \Dom^{q+1}$ is equivalent to $A_q(D^q)\subset \Dom^{q+1}$. Since $\pi_{\min}^{q+1}A_q(\Dom^q)\subset \Dom_{\min}^{q+1}\subset \Dom^{q+1}$ in any case, the requirement is that $\pi^{q+1}_{\max}A_q(D^q)\subset D^{q+1}$.

\smallskip

Let $\mathbf d=(d_0,d_1,\dotsc,d_{m-1})\in \ZN_0^m$. Denote by $\Gr_d(V)$ the Grassmannian of $d$-dimensional subspaces of a vector space $V$.  Define
\begin{equation*}
\mathfrak G_{\mathbf d}=
\set{(D^0,\dotsc,D^{m-1})\in \prod_{q=0}^{m-1}\Gr_{d_q}(\Sing^q):\a_qD^q\subset D^{q+1},\ q=0,\dotsc,m-1}.
\end{equation*}
so that for each $\mathbf D=(D^0,\dotsc,D^{m-1})\in \mathfrak G_{\mathbf d}$ we get a Hilbert complex by specifying the spaces $\Dom^q=D^q+\Dom_{\min}^q$ as domains.

\begin{proposition}
The set of Hilbert complexes associated with the $c$-elliptic complex \eqref{TheCComplex} is in one to one correspondence with
\begin{equation*}
\mathfrak G= \bigcup_{\mathbf d\in \ZN^m_0}\mathfrak G_{\mathbf d}.
\end{equation*}
For each $\mathbf d\in \ZN^m_0$, the set $\mathfrak G_{\mathbf d}$ is an algebraic subvariety of $\prod_{q=0}^{m-1}\Gr_{d_q}(\Sing^q)$.
\end{proposition}

Only the last assertion needs to be proved.

\begin{lemma}
Let $V$ be a finite-dimensional vector space and let $\a:V\to V$ be a linear map. Let $d_0$, $d_1$ be nonnegative integers. Then
\begin{equation*}
\mathscr V=\set{(X,Y)\in\Gr_{d_0}(V)\times \Gr_{d_1}(V):\a X\subset Y}
\end{equation*}
is an algebraic subvariety of $\Gr_{d_0}(V)\times \Gr_{d_1}(V)$.
\end{lemma}

\begin{proof}
Let $(X_0,Y_0)\in \Gr_{d_0}(V)\times \Gr_{d_1}(V)$. Pick bases $e_1,\dotsc,e_N$ and $f_1,\dotsc,f_N$ for $V$ so that $e_1,\dotsc,e_{d_0}$ is a basis of $X_0$ and $f_1,\dotsc,f_{d_1}$ is a basis of $Y_0$. The points $X$ of $\Gr_{d_0}(V)$ near $X_0$ are parametrized by
\begin{equation*}
X=\Span\set{ e_j+\sum_{k=d_0+1}^N x^k_j e_k:j=1,\dotsc,d_0}
\end{equation*}
and the points $Y$ of $\Gr_{d_1}(V)$ near $Y_0$ are parametrized by
\begin{equation*}
Y=\Span\set{ f_\mu+\sum_{\nu=d_1+1}^N y^\nu_\mu f_\nu:\mu=1,\dotsc,d_1}
\end{equation*}
Let $f^\mu$ denote the basis dual to the basis $f_\mu$. The annihilator of $Y$ is spanned by
\begin{equation*}
f^\mu -\sum_{\nu=1}^{d_1} y^\mu_\nu f^\nu,\quad \mu=d_1+1,\dotsc, N. 
\end{equation*}
Let the $a^\mu_j$ be so that $\a e_j= \sum_\mu \a^\mu_j f_\mu$. The condition that $\a X\subset Y$ is
\begin{equation*}
\langle f^\mu -\sum_{\nu=1}^{d_1} y^\mu_\nu f^\nu, \sum_{\mu'=1}^N (\a^{\mu'}_j+\sum_{k=d_0+1}^N x^k_j\a^{\mu'}_k)f_{\mu'}\rangle = 0
\end{equation*}
for $j=1,\dotsc, d_0$ and $\mu=d_1+1,\dotsc N$. These conditions are polynomial equations of degree at most $2$. So $\mathscr V$ is an algebraic variety.
\end{proof}

The same argument shows that for any $d\in \ZN_0$,
\begin{equation*}
\mathscr V=\set{X\in\Gr_{d}(V):\a X\subset X}
\end{equation*}
is an algebraic variety.

Now let $V=\bigoplus_{q=0}^{m-1} \Sing^q$, define $\a:V\to V$ by
\begin{equation*}
\a(\sum\phi_q)=\sum\a_q\phi_q,
\end{equation*}
and let $\pi_q:V\to V$ be the canonical projection on $\Sing^q$. Let $d=\sum_{q=0}^{m-1} d_q$. Then
\begin{equation*}
\mathscr V_c=\set{X\in \Gr_d(V):\a X\subset X}
\end{equation*}
is a variety, as is
\begin{equation*}
\mathscr V_s=\set{X\in \Gr_d(V):\pi_q X\subset X\text{ for all }q}
\end{equation*}
The elements of the latter variety are of the form $X=\bigoplus X_q$ with $X_q$ a subspace of $\Sing^q$ of some dimension. The set $\mathscr V=\mathscr V_c\cap \mathscr V_s\subset \Gr_d(V)$ is thus a variety. It splits as a disjoint union of subsets
\begin{equation*}
\mathscr V\cap \prod_{q=0}^{m-1}\Gr_{d_q'}(\Sing^q),\quad (d_0',\dotsc,d_{m-1}')\in \ZN_0^m,\ \sum d_q'=d.
\end{equation*}
It follows that
\begin{equation*}
\mathfrak G_{\mathbf d}=\mathscr V\cap \prod_{q=0}^{m-1}\Gr_{d_q}(\Sing^q)
\end{equation*}
is an algebraic subvariety of $\prod_{q=0}^{m-1}\Gr_{d_q}(\Sing^q)$.

\section{Secondary cohomology}\label{SecondaryCohomology}

Let $\mathbf D=(D^0,\dots,D^{m-1})$ with subspaces $D^q\subset \Sing^q$. The vanishing of the map $\a_q$ in \eqref{afrak} in degree $q$ removes the condition on $D^{q+1}$ in Lemma~\ref{IdealBdy} and is equivalent to the condition that
\begin{equation*}
A_{q}(\Dom_{\max}^q)\subset\Dom_{\min}^{q+1}.
\end{equation*}
Replacing the Taylor expansions \eqref{TaylorOfA} of $A_{q+1}$ and $A_q$ in $A_{q+1} A_q=0$ and collecting terms gives
\begin{gather*}
A_{q+1}^{(0)}A_q^{(0)}=0,\\
A_{q+1}^{(0)}P_q^{(1)}+P_{q+1}^{(1)}A_q^{(0)}=0,\\
A_{q+1}^{(0)}xP_q^{(2)}+P_{q+1}^{(1)}P_q^{(1)}+xP_q^{(2)}A_q^{(0)}=0.
\end{gather*}
The first of these formulas of course give the complex \eqref{gSingComplex}. The second gives a chain map of degree $1$,
\begin{equation*}
\begin{CD}
\cdots @>>> \gSing_{\sigma_0}(\Z^\wedge;E^q_\Z) @>{A_q^{(0)}}>>\gSing_{\sigma_0+\im}(\Z^\wedge;E^{q+1}_\Z) @>>> \cdots\\
@. @V{P_q^{(1)}}VV @V{P_{q+1}^{(1)}}VV @. \\
\cdots  @>>> \gSing_{\sigma_0}(\Z^\wedge;E^{q+1}_\Z) @>{-A_{q+1}^{(0)}}>>\gSing_{\sigma_0+\im}(\Z^\wedge;E^{q+2}_\Z) @>>> \cdots
\end{CD}
\end{equation*}
Write $\mathbf P_q^{(1)}$ for the induced map $\H^q_{\sigma_0}(\Z;A)\to \H^{q+1}_{\sigma_0}(\Z;A)$.

\begin{proposition}
The maps $\mathbf P_q^{(1)}$ join to give a complex
\begin{equation*}
\cdots \to \H^{q-1}_{\sigma_0}(\Z;A) \xrightarrow{\mathbf P_{q-1}^{(1)}} \H^{q}_{\sigma_0}(\Z;A) \xrightarrow{\mathbf P_q^{(1)}} \H^{q+1}_{\sigma_0}(\Z;A) \to \cdots.
\end{equation*}
\end{proposition}

\begin{proof}
We have 
\begin{equation*}
P_{q+1}^{(1)}P_q^{(1)}=-A_{q+1}^{(0)}xP_q^{(2)}-xP_q^{(2)}A_q^{(0)}
\end{equation*}
If $\mathbf u\in \H^q_{\sigma_0}(\Z;A)$ is represented by $u\in \gSing_{\sigma_0}(\Z^\wedge;E^q_\Z)$, in particular, $A_q^{(0)}u=0$, then the above formula gives that
\begin{equation*}
P_{q+1}^{(1)}P_q^{(1)}u=-A_{q+1}^{(0)}xP_q^{(2)}u
\end{equation*}
so $v=P_{q+1}^{(1)}P_q^{(1)}u$  is a coboundary in $\gSing_{\sigma_0}(\Z^\wedge;E^{q+2}_\Z)$, hence represents the zero element in $\H^{q+2}_{\sigma_0}(\Z;A)$. Thus $\mathbf P_{q+1}^{(1)} \mathbf P_q^{(1)}\mathbf  u=0$.
\end{proof}

\begin{theorem}
Let $\sigma_0\in \C$, $\gamma-1<\Im\sigma_0<\gamma$. Let $u\in \gSing_{\sigma_0}(\Z^\wedge;E^q_Z)$ be $A_q^{(0)}$-closed, $\mathbf u\in \H^q_{\sigma_0}(\Z;A)$ its class. Then $\mathbf P_q^{(1)}\mathbf u=0$ if and only if $A_q\omega u\in \Dom_{\min}^{q+1}$.
\end{theorem}

\begin{proof}
The condition $\mathbf P_q^{(1)}\mathbf u=0$ means there is $v\in \gSing_{\sigma_0-\im}(\Z^\wedge;E^q_\Z)$ such that $P_q^{(1)} u=A_q^{(0)}v$. In $\M$ near $\Z$ we have
\begin{equation*}
A_q u=A_q^{(0)}u+P_q^{(1)}u+x\tilde P_q^{(2)}u=A_q^{(0)}v+x\tilde P_q^{(2)}u=A_q v-\tilde P_q^{(1)}v+x\tilde P_q^{(2)}u.
\end{equation*}
Thus
\begin{align*}
A_q \omega u&=\omega A_q u-\im \csym(A_q)(d\omega)(u) \\
&=A_q\omega v+\omega(-\tilde P_q^{(1)}v+x\tilde P_q^{(2)}u)+\im \csym(A_q)(d\omega)(v-u).
\end{align*}
By Example \ref{DminInterference}, $A_q\omega v\in \Dom_{\min}^{q+1}$. The term $\im \csym(A)(d\omega)(v-u)$ belongs to the space $C_c^\infty(\open\M;E^{q+1})$, so it also lies in $\Dom_{\min}^{q+1}$. Finally, $\omega(-\tilde P_q^{(1)}v+x\tilde P_q^{(2)}u)\in x^{-\gamma+1}H^\infty_b\subset \Dom_{\min}^{q+1}$.

Let conversely $A_q\omega u \in \Dom^{q+1}_{\min}$. We have
$$
A_q\omega u = \omega P_q^{(1)} u + \omega x\tilde{P}_q^{(2)}u - \im\csym(A_q)(d\omega)(u),
$$
where $\omega x\tilde{P}_q^{(2)}u - \im\csym(A_q)(d\omega)(u) \in x^{-\gamma+1}H^{\infty}_b \subset \Dom^{q+1}_{\min}$,
and thus $P_q^{(1)}u \in \gSing_{\sigma_0}(\Z^{\wedge};E^{q+1}_{\Z})$ such that $\omega P_q^{(1)}u \in \Dom^{q+1}_{\min}$.
By Proposition~\ref{DminInSigmaIsExact} there exists $v \in \gSing_{\sigma_0-\im}(\Z^{\wedge};E^q_{\Z})$ such that $A_q^{(0)}v = P_q^{(1)}u$, so $\mathbf P_q^{(1)}\mathbf u = 0$.
\end{proof}

\begin{corollary}
If $\sigma_0\in \spec_b^q(A)$ with $\gamma-1<\Im\sigma_0<\gamma$ and $\sigma_0\notin \spec_b^{q+1}(A)$, then $A_q(\omega u)\in \Dom_{\min}^{q+1}$ for every representative $u$ of an element in $\H_{\sigma_0}^q(\Z;A)$. In particular, if
\begin{equation*}
\spec^{q}_b(A)\cap\spec^{q+1}_b(A)\cap\{\sigma_0\in\C;\; \gamma-1<\Im(\sigma_0)<\gamma\} = \emptyset
\end{equation*}
then $A_q\Dom_{\max}^q\subset \Dom_{\min}^{q+1}$.
\end{corollary}

\begin{corollary}
If $P_q^{(1)}=0$, then $A_q\Dom_{\max}^q\subset \Dom_{\min}^{q+1}$.
\end{corollary}


\end{document}